\documentclass[a4, 12pt]{amsart}
\usepackage{amssymb}
\usepackage{amstext}
\usepackage{amsmath}
\usepackage{amscd}
\usepackage{latexsym}
\usepackage{amsfonts}
\usepackage{color}
\usepackage{enumerate}
\usepackage{textcomp}
\usepackage[all]{xy}

\theoremstyle{plain}
\newtheorem{thm}{Theorem}[section]
\newtheorem*{thm*}{Theorem}
\newtheorem*{cor*}{Corollary}
\newtheorem*{defn*}{Definition}

\newtheorem{prop}[thm]{Proposition}
\newtheorem{lem}[thm]{Lemma}
\newtheorem{cor}[thm]{Corollary}
\newtheorem{claim}{Claim}
\newtheorem*{claim*}{Claim}

\theoremstyle{definition}
\newtheorem{defn}[thm]{Definition}
\newtheorem{ex}[thm]{Example}
\newtheorem{rem}[thm]{Remark}

\newtheorem{question}[thm]{Question}

\newtheorem{setting}[thm]{Setting}

\theoremstyle{remark}

\numberwithin{equation}{thm}
\def\Hom{\mathrm{Hom}}
\def\Max{\mathrm{Max}}

\def\Ext{\mathrm{Ext}}

\def\bbZ{\mathbb{Z}}

\def\Ker{\mathrm{Ker}}

\def\e{\mathrm{e}}
\def\m{\mathfrak m}
\def\n{\mathfrak n}

\def\p{\mathfrak p}

\def\Z{\Bbb Z}

\def\bbI{\Bbb I}

\newcommand{\rmc}{\mathrm{c}}

\newcommand{\rme}{\mathrm{e}}
\newcommand{\rmf}{\mathrm{f}}

\newcommand{\rmr}{\mathrm{r}}

\newcommand{\rmE}{\mathrm{E}}

\newcommand{\rmH}{\mathrm{H}}

\newcommand{\rmK}{\mathrm{K}}

\newcommand{\rmQ}{\mathrm{Q}}

\newcommand{\fka}{\mathfrak{a}}

\newcommand{\fkc}{\mathfrak{c}}

\newcommand{\fkm}{\mathfrak{m}}

\newcommand{\fkp}{\mathfrak{p}}
\newcommand{\fkq}{\mathfrak{q}}

\newcommand{\fkC}{\mathfrak{C}}

\def\Ann{\mathrm{Ann}}
\def\Ass{\mathrm{Ass}}

\def\GCD{\operatorname{GCD}}
\begin{document}

\setlength{\baselineskip}{15pt}
\title{Almost Gorenstein rings}
\pagestyle{plain}
\author{Shiro Goto}
\address{Department of Mathematics, School of Science and Technology, Meiji University, 1-1-1 Higashi-mita, Tama-ku, Kawasaki 214-8571, Japan}
\email{goto@math.meiji.ac.jp}
\author{Naoyuki Matsuoka}
\address{Department of Mathematics, School of Science and Technology, Meiji University, 1-1-1 Higashi-mita, Tama-ku, Kawasaki 214-8571, Japan}
\email{matsuoka@math.meiji.ac.jp}
\author{Tran Thi Phuong}
\address{Department of Information Technology and Applied Mathematics,
Ton Duc Thang University, 98 Ngo Tat To Street, Ward 19, Binh Thanh District, Ho Chi Minh City, Vietnam}
\email{sugarphuong@gmail.com}
\endgraf

\begin{abstract}
The notion of almost Gorenstein ring given by Barucci and Fr{\"o}berg \cite{BF} in the case where the local rings are analytically unramified is generalized, so that it works well also in the case where the rings are analytically ramified. As a sequel, the problem of when the endomorphism algebra $\m : \m$ of $\m$ is a Gorenstein ring is solved in full generality,  where $\m$ denotes the maximal ideal in a given Cohen-Macaulay local ring of dimension one. Characterizations of almost Gorenstein rings are given in connection with the principle of idealization. Examples are explored. 
\end{abstract}

\maketitle

{\footnotesize \tableofcontents }

\section{Introduction} 
This paper studies a special class of one-dimensional Cohen-Macaulay local rings, which we call {\it almost Gorenstein} rings. Originally, almost Gorenstein rings were introduced by V. Barucci and R. Fr{\"o}berg \cite{BF}, in the case where the local rings are analytically unramified. They developed  in \cite{BF}  a very  nice theory of almost Gorenstein rings and gave many interesting results, as well.  Our paper aims at an alternative definition of almost Gorenstein ring which we can apply also to the rings that are not necessarily analytically unramified. One of the purposes of such an alternation is to go beyond  a gap in the proof of \cite[Proposition 25]{BF} and solve in full generality the problem of when the algebra $\m : \m$ is a Gorenstein ring, where $\m$ denotes the maximal ideal in a given Cohen-Macaulay local ring of dimension one.

Before going into more details, let us fix our notation and terminology, which we maintain throughout this paper.

Let $R$ be a Cohen-Macaulay local ring with maximal ideal $\m$ and $\dim R = 1$. We denote by $\rmQ (R)$ the total quotient ring of $R$.  Let $\rmK_R$ be the canonical module of $R$. Then we say that an ideal $I$ of $R$ is {\it canonical}, if $I \ne R$ and $I \cong \rmK_R$ as $R$--modules. As is known by \cite[Satz 6.21]{HK}, $R$ possesses a canonical ideal if and only if $\rmQ (\widehat{R})$ is a Gorenstein ring, where $\widehat{R}$ denotes the $\m$--adic completion of $R$. Therefore, the ring  $R$ possesses a canonical ideal, once it is analytically unramified, that is the case where $\widehat{R}$ is a reduced ring.

Let $I$ be a canonical ideal of $R$. Then because $\operatorname{Ann}_R I = (0)$, the ideal $I$ is  $\m$--primary, and we have integers $\rme_0(I) > 0$ and $\rme_1(I)$ such that the Hilbert function of $I$ is given by the polynomial 
$$\ell_R(R/I^{n+1}) = \rme_0(I)\binom{n+1}{1} - \rme_1(I)$$ 
for all integers $n \gg 0$, where $\ell_R(M)$ denotes, for each $R$--module $M$,  the length of $M$. Let $\mathrm{r}(R) = \ell_R(\operatorname{Ext}_R^1(R/\m, R))$ be the Cohen-Macaulay type of $R$ (\cite[Definition 1.20]{HK}). Then our definition of almost Gorenstein ring is now stated as follows.

\begin{defn*} [Definition \ref{almostGor}]
We say that $R$ is an almost Gorenstein ring, if $R$ possesses a canonical ideal $I$ such that $\rme_1(I) \le \rmr(R)$.
\end{defn*}
\noindent
If $R$ is a Gorenstein ring, then we can choose any parameter ideal $Q$ of $R$ to be  a canonical ideal and  get $\rme_1(Q) = 0 < \rmr (R) = 1$. Hence every Gorenstein local ring of dimension one is an almost Gorenstein ring.

Let us explain how this paper  is organized, describing the main contents in it. In Section 2 we would like to invite the reader to revisit some well-known results, say Northcott-Rees' inequalities, on the first Hilbert coefficients $\rme_1(I)$ of $\m$--primary ideals $I$ in $R$. These results, especially those  for canonical ideals of $R$, have led us to the present research and  control the whole story of this paper. We shall discuss in Section 2  also the condition under which $R$ possesses {\it fractional} ideals $K$ such that $R \subseteq K \subseteq \overline{R}$ and $K \cong \rmK_R$ as $R$--modules.

In Section 3 we shall give  characterizations of Gorenstein rings and almost Gorenstein rings as well, according to Definition \ref{almostGor}. Our definition \ref{almostGor} of almost Gorenstein ring is rather different from the one which Barucci and Fr{\"o}berg gave in the analytically unramified case \cite[Definition-Proposition 20]{BF}. (Here notice that in Definition \ref{almostGor} $R$ is not assumed to be analytically unramified.) However, despite the difference in appearance, both the definitions are equivalent to each other in the analytically unramified case, which we will confirm in Section 3.  

In Section 4 we will explore $3$--generated numerical semigroup rings over a field and their almost Gorenstein property.  Corollary \ref{NHW} has been reported by H. Nari \cite{NNW} (see \cite{NNW2} also) at the 32-nd Symposium on Commutative Algebra in Japan (Hayama, 2010). Our research is independent of \cite{NNW, NNW2}.

In Section 5 we will study the problem of when the endomorphism algebra $\m : \m ~(\cong \Hom_R(\m,\m))$ of $\m$ is a Gorenstein ring. This is the problem which Barucci and Fr{\"o}berg wanted to solve in \cite{BF}, but Barucci finally felt there was a gap in \cite[Proof of Proposition 25]{BF}. However, we should note here that the counterexample \cite[Example p. 995]{B} given by Barucci to \cite[Proposition 25]{BF} is wrong, and the proof stated in \cite{BF} still works with our modified definition of almost Gorenstein ring, which we shall closely discuss in Section 5.

In the last section 6 we will give a series of characterizations of almost Gorenstein rings obtained by idealization (namely, trivial extension), including the following theorem.

\begin{thm*} [Theorem \ref{algorm}] Let $R \ltimes \m$ denote the idealization of $\m$ over $R$. Then the following conditions are equivalent.
\begin{enumerate}[{\rm (1)}]
\item $R\ltimes {\m}$ is an almost Gorenstein ring.
\item $R$ is an almost Gorenstein ring.
\end{enumerate}
\end{thm*}
\noindent 
This result enables us to construct infinitely many examples of analytically ramified almost Gorenstein rings that are not Gorenstein, which shows our modified definition \ref{almostGor} enriches concrete examples of almost Gorenstein rings as well as the theory.

Unless otherwise specified, in what follows, let $(R, \m)$ denote a Cohen-Macaulay local ring with $\dim R = 1$. Let $\rmQ (R)$ be the total quotient ring of $R$ and $\overline{R}$ the integral closure of $R$ in $\rmQ (R)$. For each finitely generated $R$--module $M$, let $\mu_R(M)$ denote the number of elements in a minimal system of generators for $M$. Let $v(R) = \mu_R(\m)$ and $e(R) = \rme_0({\m})$, the multiplicity of $R$ with respect to $\m$. Let $\ell_R(*)$ stand for the length.
For given fractional ideals $F_1$, $F_2$ of $R$, let $F_1:F_2=\{x\in \rmQ(R)\mid xF_2\subseteq F_1\}$. When we consider the ideal colon $\{x \in R \mid xJ \subseteq I\}$ for integral ideals $I$, $J$ of $R$, we denote it by $I:_RJ$ in order to make sure of the meaning.


\section{The first Hilbert coefficients and existence of canonical ideals}
In this section we shall summarize preliminary results, which we need throughout this paper. Some of them are known but let us note  brief proofs for the sake of completeness.

Let $R$ be a Cohen-Macaulay local ring with maximal ideal $\m$ and $\operatorname{dim} R = 1$. Let $I$ be an $\m$--primary ideal of $R$. Then there exist integers $\rme_0(I) >0$ and $\rme_1(I)$ such that $$\ell_R(R/I^{n+1}) = \rme_0(I)\binom{n+1}{1} - \rme_1(I)$$ for all integers $n \gg 0$. We assume that there exists an element $a \in I$ such that the ideal $Q=(a)$ is a reduction of $I$, i.e., $I^{r+1} = QI^r$ for some integer $r \ge 0$ (this condition is automatically satisfied, if the residue class field $R/\m$ of $R$  is infinite). We put  
$$r = \mathrm{red}_Q(I) :=  \min\{n \in \mathbb Z \mid I^{n+1}=QI^n\}.$$ 
For each integer $n\geq 0$ let $\frac{I^n}{a^n} = \{ \frac{x}{a^n} \mid x \in I^n \}$ and put $S = R[\frac{I}{a}]$ in $\rmQ (R)$. We then have $\frac{I^n}{a^n} \subseteq \frac{I^{n+1}}{a^{n+1}}$ for all $n \ge 0$. Therefore, since $S = \bigcup_{n \ge 0}\frac{I^n}{a^n}$ and $\frac{I^n}{a^n} = \frac{I^r}{a^r}$ for all $n \ge r$, we get $S = \frac{I^r}{a^r} \cong I^r$ as $R$--modules. Hence $S$ is a finitely generated $R$--module, so that  
$$R \subseteq S \subseteq \overline{R}.$$

Let $n \ge 0$ be an integer. Then, since $I^{n+1}/Q^{n+1} \cong [\frac{I^n}{a^n}]/R \subseteq S/R$, we have
\begin{eqnarray*}
\ell_R(R/I^{n+1}) &=& \ell_R(R/Q^{n+1}) - \ell_R(I^{n+1}/Q^{n+1})\\
& \ge & \ell_R(R/Q^{n+1}) - \ell_R(S/R)\\
&=& \ell_R(R/Q)\binom{n+1}{1} - \ell_R(S/R)
\end{eqnarray*}
and
$$\ell_R(R/I^{n+1}) = \ell_R(R/Q)\binom{n+1}{1} - \ell_R(S/R),$$
 if $n \ge r - 1$.
Consequently we get the following.

\begin{lem}\label{calce1}
$\rme_0(I) = \ell_R(R/Q)$ and $$0 \le \rme_1(I)=\ell_R(I^r/Q^r)=\ell_R(S/R) \le \ell_R(\overline{R}/R).$$
\end{lem}

The following result is fairly well-known. We however note a brief proof, because it controls the whole story of this paper.

\begin{prop}[cf. \cite{N}]\label{ineq} $r \le \rme_1(I)$ and $$\mu_R(I/Q) = \mu_R(I) - 1 \le \ell_R(I/Q)  = \rme_0(I) - \ell_R(R/I) \le \rme_1(I).$$ We furthermore have the following.
\begin{enumerate}[{\rm (1)}]
\item $\mu_R(I/Q) = \ell_R(I/Q)$ if and only if $\m I \subseteq Q$, i.e., $\m I = \m Q$.
\item $\ell_R(I/Q) = \rme_1(I)$ if and only if $I^2=QI$.
\end{enumerate}
\end{prop}

\begin{proof} We may assume $r > 0$, that is  $I \ne Q$. Look at  the embedding 
$$(0) \hookrightarrow I/Q\overset{a}{\hookrightarrow} I^2/Q^2\overset{a}{\hookrightarrow} \cdots \overset{a}{\hookrightarrow} I^{r-1}/Q^{r-1} \overset{a}{\hookrightarrow} I^r/Q^r \overset{\sim}{\rightarrow} I^{r+1}/Q^{r+1} \overset{\sim}{\rightarrow} \ldots$$
and we get $$r \le \ell_R(I^r/Q^r),\ \ \ell_R(I/Q) \le \ell_R(I^r/Q^r)=\ell_R(S/R)={\e}_1(I).$$
Hence $r\leq {\e}_1(I)$ and $\rme_0(I) - \ell_R(R/I) = \ell_R(R/Q) - \ell_R(R/I)=\ell_R(I/Q) \leq {\e}_1(I)$. We have the equality $\ell_R(I/Q) = \rme_1(I)$ if and only if $r=1$, i.e., $I^2 = QI$. We clearly have $\mu_R(I/Q) \le \ell_R(I/Q)$, and  $\mu_R(I/Q) = \ell_R(I/Q)$ if and only if $\m I \subseteq Q$. The latter condition is equivalent to saying that $\m I = \m Q$, since $Q$ is a minimal reduction of $I$. 
\end{proof}

\begin{rem}
Proposition  \ref{ineq} is a special case of the results which hold true for arbitrary Cohen-Macaulay local rings of positive dimension. The inequality $\ell_R(I/Q) \le \rme_1(I)$ is known as Northcott's inequality (\cite{N}), and assertion (2) of Proposition \ref{ineq} was proven by \cite{H, O} independently. The ideals $I$ satisfying the condition that $\m I \subseteq Q$ are called ideals of minimal multiplicity (\cite{G}). 
\end{rem}

The estimations given by Proposition \ref{ineq} are sharp, as we see in the following examples.

\begin{ex}\label{1.5}
 Let $k[[t]]$ be the formal power series ring over a field $k$. 
\begin{enumerate}
\item Let $R = k[[t^3, t^5, t^7]]$, $I = (t^3, t^5)$, and $Q = (t^3)$. Then $Q$ is a reduction of $I$ with $\mathrm{red}_Q(I) = 2$. Hence $S = R[\frac{I}{t^3}] = k[[t^2, t^3]]$. We have $\m I \subseteq Q$ and $\rme_0(I) = \ell_R(R/Q) = 3$, so that $$\mu_R(I/Q) = \ell_R(I/Q) = 1 < \rme_1(I) = 2,$$
because $\rme_1(I) = \ell_R(S/R) = \ell_R(k[[t^2,t^3]]/R) =2$. Therefore $$\ell_R(R/I^{n+1}) = 3\binom{n+1}{1} - 2$$ for all $n \ge 1$.

\item Let $R = k[[t^3, t^5]]$, $I = (t^5,t^9)$, and $Q = (t^5)$. Then $\mathrm{red}_Q(I) = 1$ but $\m I \not\subseteq Q$. We have $\rme_0(I) = 5$ and $$\mu_R(I/Q) = 1 < \ell_R(I/Q) = \rme_1(I)  = 2.$$ Hence $\ell_R(R/I^{n+1}) = 5\binom{n+1}{1} -2$ for all $n \ge 0$. 

\item Let $R = k[[t^3, t^7, t^8]]$, $I=(t^6, t^7)$, and $Q = (t^6)$. Then $\mathrm{red}_Q(I) = 2$. We have $\rme_0(I) = 6$, $\rme_1(I) = \ell_R(k[[t]]/R) = 4$, and $$\mu_R(I/Q) = 1 < \ell_R(I/Q) = 2 < \rme_1(I).$$ Hence $\ell_R(R/I^{n+1}) = 6\binom{n+1}{1} - 4$ for all $n \ge 1$. As for the maximal ideal $\m$ of $R = k[[t^3, t^7, t^8]]$, we have $\mathrm{red}_\fkq(\m) = 1$ where $\fkq = (t^3)$.  Hence $\rme_0(\m) = 3$, while we have $\rme_1(\m) = 2$ as $R[\frac{\m}{t^3}] = k[[t^3,t^4,t^5]]$. Therefore $$\mu_R(\m/\fkq) = \ell_R(\m/\fkq) = \rme_1(\m)$$ and $\ell_R(R/\m^{n+1}) = 3\binom{n+1}{1} - 2$ for all $n \ge 0$. 
\end{enumerate}
\end{ex}

We note a few consequences. Let $\overline \fka$ denote, for each ideal $\fka$ of $R$, the integral closure of $\fka$.

\begin{cor}\label{reme1} The following assertions hold true.
\begin{enumerate}[$(1)$]
\item Let $I$ and $J$ be $\m$--primary ideals of $R$ and suppose that $I$ contains a reduction $Q = (a)$. If $I\subseteq J\subseteq \overline I$, then ${\e}_1(I)\leq {\e}_1(J)$.
\item Suppose that $R$ is not a discrete valuation ring. Then $\e_1(Q:_R{\m}) = \rmr(R)$ for every parameter ideal $Q=(a)$ of $R$, where $\rmr (R)$ denotes the Cohen--Macaulay type of $R$. 
\end{enumerate}
\end{cor}

\begin{proof}(1) Since $I \subseteq J \subseteq \overline{I}$, $Q$ is also a reduction of $J$ and $\frac{I}{a}\subseteq \frac{J}{a}$. Hence  by Lemma \ref{calce1} we get ${\e}_1(I)= \ell_R(R\left[\frac{I}{a}\right]/R) \leq \ell_R(R\left[\frac{J}{a}\right]/R)={\e}_1(J)$.

(2) We put  $I=Q:_R{\m}$. Then $I^2=QI$ by \cite{CP}, because $R$ is not  regular. Therefore  $R[\frac{I}{a}]=\frac{I}{a}$ and we get $\e_1(I)=\ell_R(R[\frac{I}{a}]/R)=\ell_R(I/Q)=\rmr(R).$ 
\end{proof}


Let $\rmK_R$ denote the canonical module of $R$. Remember that for the $\m$--adic completion $\widehat{R}$ of $R$, the canonical module $\rmK_{\widehat{R}}$ is defined by $$\rmK_{\widehat{R}} = \Hom_{\widehat{R}}(\rmH_{\widehat{\m}}^1(\widehat{R}), E),$$ where $\rmH_{\widehat{\m}}^1(\widehat{R})$ denotes the first local cohomology module of $\widehat{R}$ with respect to $\widehat{\m}$ and $E=\rmE_{\widehat{R}}(\widehat{R}/\widehat{\m})$ the injective envelope of the $\widehat{R}$--module $\widehat{R}/\widehat{\m}$. When $R$ is not necessarily $\m$--adically complete, the canonical module $\rmK_R$ is defined to be an $R$--module such that $$\widehat{R}\otimes_R \rmK_R \cong \rmK_{\widehat{R}}$$ as $\widehat{R}$--modules (\cite[Definition 5.6]{HK}). The canonical module $\rmK_R$ of $R$  is uniquely determined (up to isomorphisms) by this condition (\cite[Lemma 5.8]{HK}) and $R$ is a Gorenstein ring if and only if $\rmK_R \cong R$ as $R$--modules (\cite[Satz 5.9]{HK}).

The fundamental theory of canonical modules was developed by the monumental book \cite{HK} of E. Kunz and J. Herzog. In what follows, we shall freely consult \cite{HK} about basic results on canonical modules (see \cite[Part I]{BH} also).

As is well-known, $R$ possesses the canonical module $\mathrm{K}_R$ if and only if $R$ is a homomorphic image of a Gorenstein ring (\cite{R}). In the present research we are interested also  in the condition for $R$ to contain {\it canonical ideals}.

Let us begin with the following.

\begin{defn}\label{def}
An ideal $I$ of $R$ is said to be a canonical ideal of $R$, if $I \ne R$ and $I \cong \rmK_R$ as $R$--modules.
\end{defn}

Here we confirm that this definition implicitly assumes the existence of the canonical module $\rmK_R$. Namely, the condition in Definition \ref{def} that $I \cong \rmK_R$ as $R$--modules should be read to mean that $R$ possesses the canonical module $\rmK_R$ and the ideal $I$ of $R$ is isomorphic to $\rmK_R$ as an $R$--module. Notice that canonical ideals are $\m$--primary, because they are faithful $R$--modules (\cite[Bemerkung 2.5]{HK}).

We then have the following result \cite[Satz 6.21]{HK}. Because it plays an important role in our argument, let us include a brief proof for the sake of completeness.

\begin{prop}[\cite{HK}]\label{existI}
The following conditions are equivalent.
\begin{enumerate}
\item [$(1)$] $\rmQ (\widehat{R})$ is a Gorenstein ring.
\item [$(2)$] $R$ contains a canonical ideal.
\end{enumerate}
Hence $R$ contains a canonical ideal, if $\widehat{R}$ is a reduced ring.
\end{prop}

\begin{proof}
(1) $\Rightarrow$ (2)
Since $\rmQ(\widehat{R})$ is a Gorenstein ring, for each ${\p}\in \Ass\widehat{R}$ we have $(\rmK_{\widehat{R}})_{\p}\cong \rmK_{\widehat{R}_{\p}}\cong \widehat{R}_{\p}$ as $\widehat{R}_{\p}$--modules (\cite[Korollar 6.2]{HK}). Hence the $\rmQ(\widehat{R})$--module $\rmQ(\widehat{R})\otimes_{\widehat{R}} \rmK_{\widehat{R}}$ is locally free of rank one, so that $$\rmQ(\widehat{R})\otimes_{\widehat{R}} \rmK_{\widehat{R}} \cong \rmQ(\widehat{R})$$ as $\rmQ(\widehat{R})$--modules, which shows that $\rmK_{\widehat{R}}$ is a fractional ideal of $\widehat{R}$, because $\rmK_{\widehat{R}}$ is a torsion-free $\widehat{R}$--module. We choose an ideal $J$ of $\widehat{R}$ so that $J\cong \rmK_{\widehat{R}}$ as $\widehat{R}$--modules. We may assume $J \ne \widehat{R}$. Let $I=J \cap R$. We then have $I\widehat{R} = J$, because $J$ is an $\widehat{\m}$--primary ideal of $\widehat{R}$, and hence $I\cong \rmK_R$ by definition, because $I\widehat{R} = J \cong \rmK_{\widehat{R}}$. Thus $I$ is a canonical ideal.

(2) $\Rightarrow$ (1)
Let $I$ be a canonical ideal of $R$. Hence $I\widehat{R} \cong \rmK_{\widehat{R}}$. Therefore, because  $I\widehat{R}$ is an $\widehat{\m}$--primary ideal of $\widehat{R}$, for every $\p\in \Ass{\widehat{R}}$ we get  $$\widehat{R}_\p = I\widehat{R}_\p \cong  (\rmK_{\widehat{R}})_{\p} \cong \rmK_{\widehat{R}_\p},$$ so that $\widehat{R}_\p$ is  a Gorenstein ring. Thus the ring $\rmQ(\widehat R)$ is Gorenstein.
\end{proof}

Let $\overline{R}$ denote the integral closure of $R$ in $\rmQ(R)$.

\begin{cor}\label{existK2}
The following conditions are equivalent.
\begin{enumerate}[{\rm (1)}]
\item There exists an $R$--submodule $K$ of $\rmQ (R)$ such that $R \subseteq K \subseteq \overline{R}$ and $K \cong \rmK_R$ as $R$--modules.
\item $R$ contains a canonical ideal $I$ and $a \in I$ such that $(a)$ is  a  reduction of $I$.
\end{enumerate}
When this is the case, every canonical ideal $I$ of $R$ contains an element which generates a reduction of $I$ and the first Hilbert coefficient $\rme_1(I)$ is independent of the choice of canonical ideals $I$. 
\end{cor}

\begin{proof} 
(1) $\Rightarrow$ (2)
Choose a regular element $a$ of $R$ so that $I = aK \subsetneq R$. Then $I$ is a canonical ideal of $R$ and  $(a) \subseteq I \subseteq a\overline{R}$. Hence $(a)$ is a reduction of $I$.

(2) $\Rightarrow$ (1)
Let $I$ be a canonical ideal of $R$. Let $a \in I$ and assume that $(a)$ is a reduction of $I$. Then, because $I$ is an ${\m}$--primary ideal, the element $a$ is regular, so that $I \subseteq \overline{(a)} = a\overline{R} \cap R$. Therefore $K = \frac{I}{a}$ is a required $R$--submodule of $\overline{R}$, which shows the implication (2) $\Rightarrow$ (1). Let $J$ be any other canonical ideal of $R$. Then, because $J \cong I$ as $R$--modules, we have a unit $\alpha$ of $\rmQ (R)$ such that $J = \alpha I$. Since $(a)$ is a reduction of $I$, the element $b = \alpha a \in J$  generates a reduction of $J$. Therefore, because $\frac{J}{b} = \frac{I}{a}$, we get $\rme_1(J) = \ell_R(R[\frac{J}{b}]/R) = \ell_R(R[\frac{I}{a}]/R) = \rme_1(I)$ by Lemma \ref{calce1}, which proves the last assertion. 
\end{proof}

As an immediate consequence, we get the following.

\begin{cor} \label{existK}
Assume that $\rmQ(\widehat{R})$ is a Gorenstein ring. If the residue class field $R/\m$ of $R$ is infinite, then there exists an $R$--submodule $K$ of $\rmQ (R)$ such that $R \subseteq K \subseteq \overline{R}$ and $K \cong \rmK_R$ as $R$--modules.
\end{cor}

\begin{rem}\label{notexist}
Corollary \ref{existK} is not true in general, unless the field $R/\m$ is infinite. For example, we look at the local ring $$R=k[[X,Y,Z]]/(X,Y) \cap (Y,Z) \cap (Z,X),$$ where $k[[X,Y,Z]]$ is the formal power series ring over a field $k$. Then $R$ is reduced and $\dim R = 1$.  We put $I=(x+y, y+z)$, where $x$, $y$, and $z$ denote the images of $X$, $Y$, and $Z$ in $R$, respectively. Then $I$ is a canonical ideal of $R$. If $k = \Z / 2\Z$,  no element of $I$  generates  a reduction of $I$, so that no $R$--submodules $K$ of $\rmQ (R)$ such that $R \subseteq K \subseteq \overline{R}$ are isomorphic to $\rmK_R$.
\end{rem}

\begin{proof}
We put $f = x+y+z$. Then $f$ is regular in $R$. Since $I \ne \m = I + (f)$, we have $f \not\in I$. A standard computation shows $$\m^2 = (x^2, y^2, z^2) = f\m = I\m = fI+ (x^2) = fI + (y^2) = fI + (z^2).$$ Hence $x^2, y^2, z^2 \not\in fI$, since $\m^2 \ne fI$ (notice  that $\mu_R(\m^2) = 3$ but $\mu_R(I) = 2$). Besides, because $I$ is a reduction of $\m$, $I$ is $\m$--primary, so that $\operatorname{Ann}_R I = (0)$. Therefore $I$ is a Cohen--Macaulay faithful $R$--module. Hence, to see that $I$ is a canonical ideal of $R$, it suffices to check that $\ell_R((0):_{I/fI}\m) = 1$ (see \cite[Korollar 6.12 and its proof]{HK}).

Let $\varphi \in (fI:_R\m) \cap I$ and write $\varphi = a(x+y) + b(y+z)$ for some $a, b\in R$. Then $x \varphi = ax^2 \in fI$ and $y \varphi = (a+b)y^2 \in fI$. Hence $a, b \in \m$, because $x^2, y^2 \not\in fI$. Therefore
$$
fI \subsetneq (fI:_R\m) \cap I \subseteq I \m = fI + (x^2).
$$
Consequently, because $\m x^2 = (x^3) = x{\cdot}fI$, we get $(fI:_R\m) \cap I = I\m  = fI + (x^2)$ and hence $\ell_R((0):_{I/fI}\m) = 1$. Thus $I$ is a canonical ideal of $R$.  

Let $k = \Bbb Z /2\Bbb Z$. Assume that $a \in I$ and $(a)$ is a reduction of $I$. Then $(a)$ is a reduction of $\m$, because $I$ is a reduction of $\m$. We write $$a = c_1x + c_2 y+ c_3 z + g$$ for some $c_1,c_2, c_3 \in k$ and $g \in \m^2 = I\m$. Let  $h = c_1x + c_2y+c_3z$. Then $h = a- g \in I$.
We have $\mathrm{red}_{(a)}(\m) = 1$, since $\mathrm{red}_{(f)}(\m) = 1$ (see Proposition \ref{ineq} and notice that assertion (2) is free of the choice of reductions $Q$ of $I$). We then have $$\m^2 = a\m \subseteq (h, g) \m \subseteq h\m + \m^3.$$ Hence $\m^2 = h\m = (c_1x^2, c_2y^2, c_3z^2)$, thanks to Nakayama's lemma.
Thus $c_i=1$ for every $i=1,2,3$, because $\mu_R(\m^2) = 3$. Hence $f=h \in I$, which is impossible. Thus no element of $I$ generates a reduction of $I$. Therefore, if $k = \Bbb Z/2\Bbb Z$,  by Corollary \ref{existK2} the ring $R$ possesses no $R$--submodules $K$ of $\rmQ (R)$ such that $R \subseteq K \subseteq \overline{R}$ and $K \cong \rmK_R$ as $R$--modules.
\end{proof}

The $R$--submodules $K$ of $\rmQ (R)$ such that $R \subseteq K \subseteq \overline{R}$ and $K \cong \rmK_R$ as $R$--modules play a  very important role in our argument. 
The following result insures the existence of those {\it fractional} ideals $K$, after enlarging the residue class field $R/\m$ of $R$ until it will be infinite, or even algebraically closed.

\begin{lem}[{\cite[AC IX,  p. 41, Corollaire]{Bo}}] \label{flat}
Let $(R,\m)$ be a Noetherian local ring with $k = R/\m$. Then for each extension $k_1/k$ of fields, there exists a flat local homomorphism $(R, \m) \to (R_1, \m_1)$ of Noetherian local rings which satisfies  the following conditions. 
\begin{enumerate}[{\rm (a)}]
\item $\m_1 = \m R_1$.
\item $R_1 / \m_1 \cong k_1$ as $k$--algebras.
\end{enumerate}
\end{lem}

We  apply Lemma \ref{flat} to our context.

\begin{prop}\label{f}
Let $R$ be a Cohen-Macaulay local ring with maximal ideal $\m$ and $\dim R = 1$. Let $k = R/\m$ and let $k_1/k$ be an extension of fields. Suppose that $\varphi : (R, \m) \to (R_1, \m_1)$ is a flat local homomorphism of Noetherian local rings such that 
\begin{enumerate}[{\rm (a)}]
\item $\m_1 = \m R_1$.
\item $R_1 / \m_1 \cong k_1$ as $k$--algebras.
\end{enumerate}
Then $R_1$ is a Cohen--Macaulay ring with $\operatorname{dim} R_1 = 1$. We furthermore have the following.
\begin{enumerate}[{\rm (1)}]
\item $\rmQ(\widehat{R_1})$ is a Gorenstein ring if and only if $\rmQ(\widehat{R})$ is a Gorenstein ring. When this is the case,  for every canonical ideal $I$ of $R$ the ideal  $IR_1$ of $R_1$ is a canonical ideal of $R_1$ and $\rme_1(IR_1) = \rme_1(I)$.
\item $\m_1 : \m_1$ is a Gorenstein ring if and only if $\m : \m$ is a Gorenstein ring.
\item Let $M, N$ be finitely generated $R$--modules. Then $M \cong N$ as $R$--modules if and only if $R_1 \otimes_R M \cong R_1 \otimes_R N$ as $R_1$--modules.
\end{enumerate}
\end{prop}

\begin{proof} 
Since the homomorphism $\varphi$ is flat and local and $R_1/\m R_1$ is a field, the ring $R_1$ is Cohen--Macaulay and $\operatorname{dim} R_1 = \operatorname{dim} R = 1$.

(1) Suppose that $\rmQ(\widehat{R})$ is a Gorenstein ring. Then by Proposition \ref{existI} we may choose a canonical ideal $I$ of $R$. Since $R_1/\m R_1$ is a Gorenstein ring and $I \cong \rmK_R$, by \cite[Satz 6.14]{HK} we get $IR_1 \cong \rmK_{R_1}$, so that $IR_1$ is a canonical ideal of $R_1$. Thus by Proposition \ref{existI} $\rmQ(\widehat{R_1})$ is a Gorenstein ring. Notice that 
$$\ell_{R_1}(R_1/I^{n+1}R_1) = \ell_R(R/I^{n+1})$$ for all integers $n \ge 0$,  because $\m_1 = \m R_1$. Therefore  $\rme_1(IR_1) = \rme_1(I)$.

Conversely, suppose that $\rmQ(\widehat{R_1})$ is a Gorenstein ring and let $\fkp \in \Ass \widehat{R}$. We choose $P \in \Ass \widehat{R_1}$ so that $\fkp = P \cap R$. Then, thanks to the flat descent, $\widehat{R}_{\fkp}$ is a Gorenstein ring, because $\widehat{R_1}_P$ is a Gorenstein ring and the local homomorphism $\widehat{R}_{\fkp} \to \widehat{R_1}_P$ induced from the flat homomorphism $\widehat{\varphi} : \widehat{R} \to \widehat{R_1}$ remains flat. Thus $\widehat{R}_\fkp$ is a Gorenstein ring for every $\fkp \in \Ass \widehat{R}$ and hence $\rmQ(\widehat{R})$ is a Gorenstein ring.

(2) Let $A=\m : \m$ and $B=R_1 \otimes_R A$.  Then $B$ is $A$--flat and $B \cong \m_1 :\m _1$ as $R_1$--algebras. Hence by the flat descent,  $A$ is a Gorenstein ring, if $B$ is a Gorenstein ring. Conversely, suppose that $A$ is a Gorenstein ring. Let $N$ be a maximal ideal in $B$ and we must show that $B_N$ is a Gorenstein ring. Let $M=N \cap A$. Then $M$ is a maximal ideal in $A$, because $$M \cap R = (N \cap R_1) \cap R = \m_1 \cap R = \m$$ (notice that $B$ (resp. $A$) is a module-finite extension of $R_1$ (resp. $R$)). Consequently, in order to see that $B_N$ is a Gorenstein ring, passing to the flat local homomorphism $A_M \to B_N$, it suffices to show that $B/MB$ is a Gorenstein ring.

We now look at the isomorphisms
\begin{align*}
B/MB &\cong R_1 \otimes_R A/M\\
&\cong k_1 \otimes_{R_1}(R_1 \otimes_R A/M) \\
&\cong k_1 \otimes_{k}(k \otimes_R A/M) \\
&\cong k_1 \otimes_{k} A/M.
\end{align*}
Since the field $A/M$ is a finite extension of $k$, we get $$A/M \cong k[X_1, X_2, \ldots , X_n] / (\xi_1, \xi_2, \ldots , \xi_n)$$ where $\xi_1, \xi_2, \ldots , \xi_n$ denotes  a regular sequence in the polynomial ring $k[X_1, X_2, \ldots , X_n]$ over the field $k$.  Therefore $$B/MB = k_1 \otimes_k A/M = k_1[X_1, X_2, \ldots , X_n] / (\xi_1, \xi_2, \ldots , \xi_n)$$ is a Gorenstein ring, because $\xi_1, \xi_2, \ldots , \xi_n$ forms  a regular sequence also in the polynomial ring $k_1[X_1, X_2, \ldots , X_n]$. Thus $B/MB$ is a Gorenstein ring for every maximal ideal $M $ in $A$, so that $B$ is a Gorenstein ring. 

(3) This assertion holds true without the assumption that $R$ is Cohen--Macaulay and $\operatorname{dim} R = 1$. See \cite[Proposition (2.5.8)]{EGA} for the proof.

\end{proof}

\begin{cor}\label{inv}
Suppose that $\rmQ (\widehat{R})$ is a Gorenstein ring. Then the first Hilbert coefficient $\rme_1(I)$  of $I$ is independent of the choice of canonical ideals $I$ of $R$.
\end{cor}

\begin{proof} After enlarging the residue class field $R/\m$ of $R$,  by Proposition \ref{f} (1) we may assume that the field $R/\m$ is infinite. Hence the assertion readily follows from Corollary \ref{existK2}.
\end{proof}

Proposition \ref{f} is sufficiently general for our purpose, since we need exactly the fact that the Gorenstein property of $\rmQ (\widehat{R})$ is preserved after enlarging the residue class field. We actually do not know whether the property in the ring $R$ of being analytically unramified is preserved after enlarging the residue class field. 

Let us note the following.

\begin{question}\label{question}
Let $R$ be a Cohen-Macaulay local ring with maximal ideal $\m$ and $\dim R=1$. Let $k_1/k$ be an extension of fields where $k = R/\m$. Suppose that $\widehat{R}$ is a reduced ring. In this setting, can we always choose a flat local homomorphism $(R, \m) \to (R_1, \m_1)$ of Noetherian local rings so that the following three conditions are satisfied?
\begin{enumerate}[{\rm (a)}]
\item $\m_1 = \m R_1$.
\item $R_1 / \m_1 \cong k_1$ as $k$--algebras.
\item $\widehat{R_1}$ is a reduced ring.
\end{enumerate}
\end{question}

\section{Almost Gorenstein rings} 
In this section we define almost Gorenstein rings and give  characterizations.

Let $R$ be a Cohen-Macaulay local ring with maximal ideal $\m$ and $\operatorname{dim} R = 1$.  

\begin{defn}\label{almostGor}
We say that $R$ is an almost Gorenstein ring, if $R$ possesses a canonical ideal $I$ such that $\rme_1(I) \le \rmr(R)$. 
\end{defn}

This definition is well-defined, because by  Corollary \ref{inv} the value $\rme_1(I)$ is independent of the choice of canonical ideals $I$. If $R$ is a Gorenstein ring, one can choose any parameter ideal $Q$ of $R$ to be a canonical ideal, so that $\rme_1(Q) = 0 < 1 = \rmr (R)$. Hence every one--dimensional Gorenstein local ring is almost Gorenstein.

Before going ahead, let us note basic examples of almost Gorenstein rings which are not Gorenstein. See Section 4 for more examples of $3$--generated numerical semigroup rings.

\begin{ex}\label{ex} Let $k$ be a field. 
\begin{enumerate}[{\rm (1)}]
\item We look at the rings $R_1 = k[[t^3, t^4, t^5]]$, $R_2= k[[X,Y,Z]]/ (X,Y) \cap (Y, Z) \cap (Z, X)$, and $R_3= k[[X,Y,Z,W]]/ (Y^2, Z^2, W^2, YW, ZW, XW- YZ)$, where $k[[t]]$, $k[[X,Y,Z]]$, and $k[[X, Y, Z, W]]$ denote the formal power series rings over $k$. Then these rings $R_1, R_2$, and $R_3$ are  almost Gorenstein rings with $\rmr (R_1) = \rmr (R_2) = 2$ and $\rmr (R_3) = 3$. The ring $R_1$ is an integral domain, $R_2$ is a reduced ring but not an integral domain, and $R_3$ is not a reduced ring.
\item Let $a \ge 3$ be an integer and put $R=k[[t^a, t^{a+1}, t^{a^2-a-1}]]$. Then 
$
\rme_1(I) = \frac{a(a-1)}{2} - 1
$
for  canonical ideals $I$ of $R$. Since $\rmr(R) = 2$, $R$ is an almost Gorenstein ring if and only if $a=3$. This example suggests that  almost Gorenstein rings are rather rare.
\end{enumerate}
\end{ex}

We note the following.

\begin{prop}\label{ff}
Let $\varphi : (R, \m) \to (R_1, \m_1)$ be a flat local homomorphism of Noetherian local rings and assume that $\m_1 = \m R_1$. Then the following conditions are equivalent.
\begin{enumerate}[$(1)$]
\item $R_1$ is an almost Gorenstein ring.
\item $R$ is an almost Gorenstein ring.
\end{enumerate}
When this is the case, $\rmr ({R_1}) = \rmr (R)$ and for every canonical ideal $I$ of $R$, $IR_1$ is a canonical ideal of $R_1$ with $\rme_1(IR_1) = \rme_1(I)$.
\end{prop}

\begin{proof}
Thanks to Proposition \ref{f} (1), we may assume that $\rmQ (\widehat{R})$ is a Gorenstein ring. Let $I$ be a canonical ideal of $R$. Then by Proposition \ref{f} (1) $IR_1$ is a canonical ideal of $R_1$ with $\rme_1(IR_1) = \rme_1(I)$, while $\rmr (R_1) = \mu_{R_1}(IR_1) = \mu_R(I) = \rmr(R)$ (\cite[Satz 6.10]{HK}). Hence the equivalence of conditions (1) and (2)  follows from Definition \ref{almostGor}.
\end{proof}

We now develop the theory of almost Gorenstein rings. For this purpose let us maintain the following setting throughout this section. Thanks to  Lemma \ref{flat}, Proposition \ref{ff}, and Corollary  \ref{existK}, we may assume this setting, after enlarging the residue class field $R/\m$ of $R$ to be infinite.

\begin{setting}\label{setting}
Let $K$ be an $R$--submodule of $\rmQ(R)$ such that $R \subseteq K \subseteq \overline{R}$ and $K \cong \rmK_R$ as $R$--modules. Let $S=R[K]$ and $\fkc = R:S$ the conductor of $S$. We choose a regular element $a \in \m$ so that $a K \subsetneq R$ and put $I=aK$, $Q=(a)$. 
\end{setting}

Notice that $Q$ is a reduction of the canonical ideal $I$ of $R$ and $S = R[\frac{I}{a}]$.

We begin with the following.

\begin{lem}\label{3.4}
\begin{enumerate}[{\rm (1)}]
\item Let $T$ be a subring of $\rmQ (R)$ such that $K \subseteq T$ and $T$ is a finitely generated $R$--module. Then $R:T = K : T$. 
\item $\fkc = K:S$ and $\ell_R(R/\fkc) = \ell_R(S/K)$.
\item $\ell_R(I/Q) = \ell_R(K/R)$ and $\ell_R(S/R) = \ell_R(R/\fkc) + \ell_R(I/Q)$.
\end{enumerate}
\end{lem}

\begin{proof}
For each subring $T$ of $\rmQ (R)$ such that $K \subseteq T$ and $T$ is a finitely generated $R$--module, we have $$K:T = K:KT = (K:K):T = R:T,$$ since $R = K : K$ (\cite[Bemerkung 2.5]{HK}). Therefore, taking $T = S$, we get $\ell_R(R/\fkc) = \ell_R(R/(K:S))$, while $$ \ell_R(R/(K:S)) = \ell_R([K:(K:S)]/(K:R)) = \ell_R(S/K),$$ thanks to the canonical duality (\cite[Bemerkung 2.5]{HK}). Thus $\ell_R(R/\fkc) = \ell_R(S/K)$. Since $K=\frac{I}{a}$, we get $\ell_R(I/Q) = \ell_R(K/R)$, so that 
\begin{eqnarray*}\label{K/A}
\ell_R(S/R)& = &\ell_R(S/K) + \ell_R(K/R)\\
&=& \ell_R(R/\fkc) + \ell_R(I/Q)\nonumber.
\end{eqnarray*}
\end{proof}

Since $\mu_R(I) = \rmr (R)$ (\cite[Satz 6.10]{HK}), combining Proposition  \ref{ineq} with Lemma \ref{3.4}, we get the following, which is the key for our argument.

\begin{prop}\label{canonineq}
$0 \le \rmr(R)-1 = \mu_R(I)- 1 \leq \ell_R(I/Q) \le \rme_1(I) = \ell_R(R/\fkc) + \ell_R(I/Q).$
\end{prop}

First of all let us note a characterization of Gorenstein rings.

\begin{thm} \label{gor}
The following conditions are equivalent.
\begin{enumerate}[{\rm (1)}]
\item $R$ is a Gorenstein ring.
\item $K = R$.
\item $S = K$.
\item $S = R$.
\item $\ell_R(S/R) = \ell_R(R/\fkc)$.
\item $I^2=QI$.
\item $\rme_1(I) = 0$.
\item $\rme_1(I) = \rmr(R)-1$.
\end{enumerate} 
\end{thm}

\begin{proof}
By Propositions \ref{ineq} and \ref{canonineq} we have $I^2 = QI$ if and only if  $\ell_R(I/Q) = \rme_1(I)$ if and only if $R = \fkc$, i.e., $R=S$. When this is the case, we get $K = R$, so that $R$ is a Gorenstein ring. If $R$ is a Gorenstein ring, then the canonical ideal $I$ is principal. Therefore $I = Q$, i.e., $K = R$, and so we certainly have  $I^2 = QI$ and $S = R[K] = R$. Similarly, if $S=K$, then $S= S: S = K : K = R$, so that $R$ is a Gorenstein ring. Thus conditions (1), (2), (3), (4), and (6) are equivalent.  Since $\rme_1(I) = \ell_R(S/R)$ by Lemma \ref{calce1}, by Proposition \ref{canonineq} condition (5) is equivalent to saying that $I = Q$, i.e., $K = R$.
If $R$ is a Gorenstein ring, then $I = Q$, so that $\rme_1(I) = 0 =  \rmr (R) - 1$. If $\rme_1(I) = \rmr (R) - 1$, then by Proposition \ref{canonineq} $\ell_R(I/Q) = \rme_1(I)$. Therefore $R$ is a Gorenstein ring. If $\rme_1(I) = 0$, then $\rmr (R) = 1$ by Proposition \ref{canonineq}, so that  $R$ is a Gorenstein ring. 
\end{proof}

\begin{cor}\label{notgor} The following assertions hold true.
\begin{enumerate}
\item[$(1)$]Suppose that $R$ is not a Gorenstein ring. Then $K:\m \subseteq S$.
\item[$(2)$]$S$ is a Gorenstein ring if and only if $\fkc^2 = a \fkc$ for some $a \in \fkc$.
\end{enumerate}
\end{cor}

\begin{proof}
(1) We have $K:\m \subseteq S$, since $\ell_R((K : \m)/K) = 1$ (\cite[Satz 3.3]{HK}) and $K \subsetneq S$ by Theorem \ref{gor}.

(2) By \cite[Satz 5.12]{HK} $\fkc S_M$ is a canonical ideal of $S_M$ for each maximal ideal $M$ of $S$, because $\fkc = K:S \cong \Hom_R(S,\rmK_R)$ by Lemma \ref{3.4}. Therefore, if $S$ is a Gorenstein ring, the ideal $\fkc$ of $S$ is locally free of rank one, so that $\fkc \cong S$ as $S$--modules. Hence $\fkc = aS$ for some $a \in \fkc$ and therefore $\fkc^2 = a \fkc$. Conversely, suppose that $\fkc^2 = a\fkc$ for some $a \in \fkc$. Then $S$ is a Gorenstein ring, thanks to  Theorem \ref{gor} which we apply to the local rings $S_M$ with $M \in \Max S$.
\end{proof}

Let us add the following.

\begin{prop}\label{rego1} The following conditions are equivalent.
\begin{enumerate}
\item[$(1)$] $R$ is a Gorenstein ring. 
\item[$(2)$] $\Ext_R^1(S,R) = (0)$.
\item[$(3)$] There exists a subring $T$ of $\rmQ (R)$ such that $K \subseteq T$, $T$ is a finitely generated $R$--module, and $\Ext_R^1(T, R) = (0)$.
\item[$(4)$] $\Ext_R^1(T, R) = (0)$ for every subring $T$ of $\rmQ (R)$ such that $K \subseteq T$ and $T$ is a finitely generated $R$--module. 
\end{enumerate}
\end{prop}

\begin{proof}
(1) $\Rightarrow$ (4) Remember that by \cite[Korollar 6.8]{HK} $\Ext_R^1(M,R) = (0)$ for every Cohen-Macaulay $R$--module $M$ with $\operatorname{dim}_RM = 1$, since $R$ is a Gorenstein ring.

(4) $\Rightarrow$ (2) and  (2) $\Rightarrow$ (3) These implications are clear.

(3) $\Rightarrow$ (1)
We apply the functor $\Hom_R(T,*)$ to the exact sequence $$0 \to R \overset{\iota}{\to} K \to K/R \to 0$$ and get the exact sequence $$0 \to R : T \overset{\iota}{\to} K: T \to \Hom_R(T, K/R) \to 0$$
of $R$--modules, where $\iota's$ denote the inclusion. We then have $\Hom_R(T, K/R)  = (0)$, because 
$R : T = K: T$ by Lemma \ref{3.4} (1). This asserts that   $K/R = (0)$, since $T$ is a finitely generated $R$--module such that $T \ne (0)$ and $\ell_R(K/R) < \infty$. Hence $R$ is a Gorenstein ring. 
\end{proof}

As a direct consequence of Proposition \ref{rego1}, we are able to recover the following result of C. J. Rego \cite{Rego}.

\begin{cor}[{\cite[Theorem]{Rego}}]\label{rego2} Suppose that $R$ is analytically unramified and let $\fkC = R : \overline{R}$ be the conductor of $\overline{R}$. Assume that $\Ext_R^1(\fkC,R) = (0)$. Then $R$ is a Gorenstein ring.
\end{cor}

\begin{proof}
We have $\fkC \cong \overline{R}$ as $\overline{R}$--modules, since $\overline{R}$ is a principal ideal ring. Therefore by Proposition \ref{rego1} $R$ is a Gorenstein ring, because $\Ext_R^1(\overline{R},R) = \Ext_R^1(\fkC, R) = (0)$ and $\overline{R}$ is a finitely generated $R$--module.
\end{proof}

We now give a characterization of almost Gorenstein rings. The following result is exactly the same as the definition of almost Gorenstein ring that Barucci and Fr{\"o}berg \cite{BF} gave in the case where  the rings $R$ are analytically unramified.

\begin{thm}\label{algor}
$R$ is an almost Gorenstein ring if and only if $\fkm K \subseteq R$, i.e., $\m I = \m Q$. When this is the case, $\m S \subseteq R$.
\end{thm}

\begin{proof} Suppose that  $R$ is an almost Gorenstein ring. If $\ell_R(I/Q) = \rme_1(I)$, then $I^2 = QI$ by Proposition \ref{ineq} (2), so that $R$ is a Gorenstein ring by Theorem \ref{gor}. If $\ell_R(I/Q) < \rme_1(I)$, then we have $\rmr (R) - 1 = \ell_R(I/Q)$, because $\rmr (R) - 1 \le \ell_R(I/Q) < \rme_1(I) \le \rmr (R)$. Therefore $\m I = \m Q$ by Proposition \ref{ineq} (1). Hence $\m I^n = \m Q^n$ for all $n \in \Bbb Z$, so that $\m S \subseteq R$, because $S =\frac{I^n}{a^n}$ for $n \gg 0$. Conversely, suppose that $\m K \subseteq R$ and we will show $R$ is an almost Gorenstein ring. We may assume that  $R$ is not a Gorenstein ring. Let $J = Q :_R \m$. Then $J^2=QJ$ by \cite{CP}, since $R$ is not a regular local ring. Therefore $I \subseteq J \subseteq \overline{I}$, so that $\rme_1(I) \le \rme_1J) = \rmr (R)$ by Corollary \ref{reme1}. Hence $R$ is an almost Gorenstein ring.
\end{proof}

Since $\m K \subseteq \m \overline{R}$, we readily have the following.

\begin{cor}\label{3.11}
If $\m \overline{R} \subseteq R$, then $R$ is an almost Gorenstein ring.
\end{cor}

We explore an example.

\begin{ex} \label{exaln} Let $e \ge 3$ be an integer. We look at the local ring $$R=k[[t^{e}, t^{e+1}, \ldots, t^{2e-1}]]$$
in the formal power series ring $k[[t]]$ over a filed $k$. Then $\m^2 = t^e\m$ and $I = (t^{e}, t^{e+1}, \ldots, t^{2e-2})$ is a canonical ideal of $R$ with $(t^e)$ a reduction. 
Choose $K = \frac{I}{t^e} = \sum_{i=0}^{e-2}Rt^i$. Then $S = R[K]=k[[t]]$ and $\rme_1(I) = \ell_R(k[[t]]/R) = e-1$ by Lemma \ref{calce1}. Because $\overline{R} = k[[t]]$ and $R:k[[t]] = \m$, by Corollary \ref{3.11} $R$ is an almost Gorenstein ring. We have $\m : \m = k[[t]]$. Hence $S = \m : \m$ is a Gorenstein ring.
\end{ex}

Let $A$ be a commutative ring and assume that $A$ contains a field of characteristic $p > 0$. Let $F : A \to A, ~ F(a) = a^p$ be the Frobenius map. We denote $A$ by $B$ when we regard $A$ as an $A$--algebra via the ring homomorphism $F$. Then we say that $A$ is $F$--pure, if the homomorphism $F : A \to B$ is a split monomorphism of $A$--modules, that is there exists  an additive map $G : A \to A$ such that $G(a^pb) = aG(b)$ for all $a, b \in A$ and $G{\cdot}F = 1_A$. With this notation we have the following.

\begin{cor}
Suppose that $R$ is complete and contains a field of positive characteristic $p > 0$. Then $R$ is an almost Gorenstein ring, if $R$ is $F$--pure.
\end{cor}

\begin{proof}
Let $F : R \to R, ~F(a) = a^p$ and let $f : \rmQ (R)/R \to \rmQ (R)/R$, $f(\overline{a}) = \overline{a^p}$, where $\overline{x}$ denotes for each $x \in R$  the image of $x$ in $\rmQ (R)/R$. Then $\overline{R}/R$ is stable under the action of $f$, i.e., $f(\overline{R}/R) \subseteq \overline{R}/R$. Notice that the  map $f$ is injective, since $R$ is $F$--pure. In fact, let $x = \frac{b}{a} \in \rmQ (R)$ with $a, b \in R$ such that $a$ is a non-zerodivisor in $R$ and assume that $x^p = \frac{b^p}{a^p} \in R$. Then, since $b^p \in  a^pR$ and since $R$ is $F$--pure, we get $b \in aR$, so that $x \in R$. Therefore, since $\m ^{\ell}{\cdot}(\overline{R}/R) = (0)$ for some $\ell \gg 0$ (remember  that $\widehat{R} = R$ and that $R$ is a reduced ring, since $R$ is $F$--pure), we have $f^{\ell}(\m{\cdot}(\overline{R}/R)) = (0)$ for all $\ell \gg 0$, so that $\m{\cdot}(\overline{R}/R) = (0)$, because $f$ is an injective map. Hence $R$ is an almost Gorenstein ring by Corollary \ref{3.11}.
\end{proof}

We need the following. Assertions (2) and (3) are fairly well-known but let us include brief proofs for the sake of completeness.

\begin{lem}\label{DVR} The following assertions hold true.
\begin{enumerate}[{\rm (1)}]
\item $\ell_R(I^2/QI) = \ell_R(R/(R:K)) \le \ell_R(S/K)$.
\item $\ell_R((R:\m)/R) = \rmr (R)$.
\item $R$ is a discrete valuation ring, if $\m : \m \subsetneq R : \m$. 
\end{enumerate}
\end{lem}

\begin{proof}
(1) Notice that $K \subseteq KK \subseteq S$. Then, since $ (KK)/K\overset{a^2}{\cong} I^2/QI$ as $R$--modules,  we get $\ell_R(I^2/QI) = \ell_R((KK)/K)  \le \ell_R(S/K)$, while $\ell_R((KK)/K) = \ell_R((K:K)/(K:KK))$ by the canonical duality. Hence the result follows, because $$R = K : K \ \ \text{and}\ \ K : KK = (K: K):K = R : K.$$

(2) Taking the $R$--dual of the exact sequence
$0 \to \m \to R \to R/\m \to 0,$
we get the exact sequence $$0 \to R \to R : \m \to \Ext_R^1(R/\m, R) \to 0.$$ Hence $$\ell_R((R:\m)/R) = \ell_R(\Ext_R^1(R/\m, R)) = \rmr (R).$$

(3) Choose $f \in R : \m \setminus \m : \m$. Then $\m f \subseteq R$ but $\m f \not\subseteq \m$. Hence $\m f = R$, so that $\m \cong R$. Thus $R$ is a discrete valuation ring.
\end{proof}

As a consequence of Theorems \ref{gor} and \ref{algor}, we get the following characterization of almost Gorenstein rings which are not Gorenstein. Condition (3) in Theorem \ref{algorcor} is called Sally's equality. $\m$--primary ideals  satisfying Sally's equality are known to enjoy very nice properties (\cite{GNO, S2, V}), where the ideals are not necessarily canonical ideals and the rings  need not  be of dimension one. For instance, the fact that condition (3) in Theorem \ref{algorcor} implies  both the condition (5) and assertion (a) is due to \cite{S2}.

\begin{thm} \label{algorcor}
The following conditions are equivalent.
\begin{enumerate}[{\rm (1)}]
\item $R$ is an almost Gorenstein ring but not a Gorenstein ring.
\item $\rme_1(I) = \rmr(R)$.
\item $\rme_1(I) = \rme_0(I) - \ell_R(R/I) + 1$.
\item $\ell_R(S/K) = 1$, i.e., $S=K:\m$.
\item $\ell_R(I^2/QI) = 1$.
\item $\m : \m = S$ and $R$ is not a discrete valuation ring.
\end{enumerate}
When this is the case, we have the following.
\begin{enumerate}[{\rm (a)}]
\item $\mathrm{red}_Q(I) = 2$.
\item $\ell_R(R/I^{n+1}) = (\rmr(R) + \ell_R(R/I) -1) \binom{n+1}{1} - \rmr(R)$ for all $n \ge 1$.
\item Let $G = \bigoplus_{n\ge 0} I^n / I^{n+1}$ be the associated graded ring of $I$ and $M= \m G + G_+$ the graded maximal ideal of $G$. Then $G$ is a Buchsbaum ring with $\bbI(G) = 1$, where $\bbI(G)$ stands for the Buchsbaum invariant of $G$.
\end{enumerate}
\end{thm}

\begin{proof} 
(1) $\Leftrightarrow$ (2)  This follows from the fact that $\rmr (R) - 1 \le \rme_1(I)$ (Proposition \ref{canonineq}). Remember that by Theorem \ref{gor}  $R$ is a Gorenstein ring if and only if $\rme_1(I) = \rmr (R) - 1$ and that $R$ is an almost Gorenstein ring if and only if $\rme_1(I) \le \rmr(R)$.   

(1) $\Leftrightarrow$ (3) $\Leftrightarrow$ (4)
We have by Lemma \ref{3.4} and Proposition \ref{canonineq} $$\ell_R(S/K) = \ell_R(R/\fkc) = \rme_1(I) - \ell_R(I/Q) = \rme_1(I) - \rme_0(I)+\ell_R(R/I).$$ Therefore, condition (3) is equivalent to saying that $\ell_R(S/K) = 1$, i.e., $\ell_R(R/\fkc) = 1$. The last condition says that $\fkc = \m$, i.e., $\m S \subseteq R$ but $S \ne R$, or equivalently, $R$ is an almost Gorenstein ring but not a Gorenstein ring (Theorems \ref{gor} and \ref{algor}). Remember that  $\ell_R((K:\m)/K = 1$ (\cite[Satz 3.3]{HK}) and that  by Corollary \ref{notgor} (1) $K : \m \subseteq S$, if $R$ is not a Gorenstein ring. Then, because $R$ is not a Gorenstein ring if $S \ne K$ (Theorem \ref{gor}), we get that  $\ell_R(S/K) = 1$ if and only if $S = K : \m$.

(4) $\Rightarrow$ (5)
By Theorem \ref{gor} $R$ is not a Gorenstein ring, so that $I^2 \ne QI$  and hence   $\ell_R(I^2/QI) = 1$, because  $\ell_R(I^2/QI)\le \ell_R(S/K)$ by Lemma \ref{DVR} (1).

(5) $\Rightarrow$ (1) By Lemma \ref{DVR} (1) we have $R:K = \m$. Therefore $\m K \subseteq R$ and $K \ne R$, so that $R$ is an almost Gorenstein ring but not a Gorenstein ring.

(1) $\Rightarrow$ (6)
Suppose that $R$ is an almost Gorenstein ring but not a Gorenstein ring.
Then $R$ is not a discrete valuation ring, $S \ne R$, and $\m S \subseteq R$. Hence $$R \subsetneq S \subseteq R:\m = \m:\m$$ by Lemma \ref{DVR} (3). Since  $\ell_R(S/R) = \rme_1(I) = \rmr(R) = \ell_R((R:\m)/R)$ (thanks to  Lemma \ref{calce1}, the equivalence of conditions (1) and (2), and Lemma \ref{DVR} (2)), we get $S=\m : \m$.

(6) $\Rightarrow $ (2) By Lemma \ref{DVR} (3) we have $S = \m : \m = R : \m$. Therefore $\rme_1(I) = \ell_R(S/R) = \ell_R((R:\m)/R) = \rmr(R)$ by Lemma \ref{calce1} and Lemma \ref{DVR} (2).   

Let us prove the last assertions. We put $J = Q:_R\m$. Hence $I \subseteq J$, because $R$ is an almost Gorenstein ring.

(a) We have $\ell_R(I/Q) = \rmr(R) -1$ by Proposition \ref{ineq} (1). Hence $\ell_R(J / I) = 1$, because $\ell_R(J/Q) = \rmr (R)$. Therefore $I^3=QI^2$ by \cite[Proposition 2.6]{GNO}, since 
$J^2=QJ$ by \cite{CP}. Thus $\mathrm{red}_Q(I) = 2$ by Theorem \ref{gor}, because $R$ is not a Gorenstein ring.  

(b) This is clear.

(c) Let $[\rmH_M^0(G)]_0$ denote the homogeneous component of the graded local cohomology module $\rmH_M^0(G)$ with degree $0$. Then, thanks to the fact $I^3 = aI^2$, a direct computation shows $$\rmH_M^0(G) = [\rmH_M^0(G)]_0 \cong (I^2:_R a) / I.$$ We want to see that $J=I^2:_R a$. Let $x \in I^2:_R a$. Then $ax \in I^2 \subseteq J^2 = aJ$. Hence $x \in J$, so that we have $$I \subseteq I^2:_R a \subseteq J = (a):_R\m.$$

\begin{claim}
$I \ne I^2:_Ra$.
\end{claim}

\begin{proof}[Proof of Claim 1]
Suppose that $I = I^2:_Ra$. Then, since $I^2 \subseteq \m I \subseteq (a)$, we have $$I^2 = a(I^2:_Ra) = aI.$$ Hence $R$ is a Gorenstein ring by Theorem \ref{gor}, which is impossible.
\end{proof}
Thanks to Claim 1, we get $I \subsetneq I^2:_Ra$, so that $J = I^2:_R a$, since $\ell_R(J/I) = 1$. Thus  $\rmH_M^0(G) = [\rmH_M^0(G)]_0 \cong J/I \cong R/\m$. Hence $G$ is a Buchsbaum ring with $\bbI(G) = 1$.
\end{proof}

The following result is a consequence of Theorem \ref{algorcor}.

\begin{cor}\label{3.12} The following assertions hold true.
\begin{enumerate}[{\rm (1)}]
\item $\rme_1(I) \ne 1$.
\item If $\rme_1(I) \le 3$, then $R$ is an almost Gorenstein ring.
\item $\rme_1(I) \ne \rmr(R) +1$.
\end{enumerate}
\end{cor}

\begin{proof} (1)
Suppose $\rme_1(I) = 1$. Then, since $R$ is not a Gorenstein ring, we get $$0 \le \rmr (R) -1 \le \ell_R(I/Q) < \rme_1(I) = 1,$$ which yields $\rmr (R) = 1$ and hence $R$ is a Gorenstein ring. This is absurd.

(2) By (1) we may assume  $\rme_1(I) \ge 2$. If $\rme_1(I) = 2$, we get $\rmr (R) \ge 2$ and $$0 < \rmr (R) -1 \le \ell_R(I/Q) < \rme_1(I) = 2.$$ Hence $\rmr (R) -1 = \ell_R(I/Q) = 1$, so that $R$ is an almost Gorenstein ring. Suppose $\rme_1(I) = 3$. If $R$ is not an almost Gorenstein ring, we get $$0 < \rmr (R) -1 < \ell_R(I/Q) < \rme_1(I) = 3.$$ Therefore $\rme_1(I) = \ell_R(I/Q) + 1 = \rme_0(I) - \ell_R(R/I) + 1$. Hence by Theorem \ref{algorcor} $R$ is an almost Gorenstein ring. This is absurd.

(3) Suppose that ${\rme}_1(I) = \rmr(R) + 1$. Then $R$ is not an almost Gorenstein ring and hence $\rmr(R)-1<\ell_R(I/Q) < \rme_1(I) = \rmr (R) + 1$. Therefore $\rme_1(I) =\rme_0(I) - \ell_R(R/I) +  1$, so that  by Theorem \ref{algorcor} $R$ is an almost Gorenstein ring, which is absurd.
\end{proof}

\begin{rem} Assertion (2) in Corollary  \ref{3.12} is no longer true, if $\rme_1(I) = 4$. For example, we look at Example \ref{1.5} (3). Then the ideal $I$ of Example \ref{1.5} (3) is a canonical ideal of $R$ with $\rme_1(I) = 4$ (see Example \ref{4.3} (1)). The ring $R$ is not an almost Gorenstein ring, because $\mu_R(I) -1 = 1 < \ell_R(I/Q) = 2$.
\end{rem}


\section{Almost Gorenstein property of $3$--generated numerical semigroup rings}
Let $k$ be a field. In this section we explore semigroup rings $k[H]$ of $3$--generated numerical semigroups $H$.

Let $a_1, a_2, a_3 \in \mathbb Z$ and assume that $0<a_1< a_2< a_3$ with $\GCD (a_1, a_2, a_3) = 1$. Let $H$ be the numerical semigroup generated by $a_1, a_2, a_3$, that is $$H = \left<a_1, a_2, a_3\right> := \{c_1a_1 + c_2a_2 +c_3a_3 \mid 0 \le c_i \in \bbZ \}.$$ Let $k[t]$ denote the polynomial ring and put $T = k[t^{a_1}, t^{a_2}, t^{a_3}]$. Then $T$ is a one-dimensional graded ring with $\overline{T} = k[t]$, where $\overline{T}$ stands for the normalization of $T$. Let $M = (t^{a_1}, t^{a_2}, t^{a_3})$ denote the maximal ideal of $T$ generated by ${t^{a_i}}'s$. In this section we explore the local ring $R=T_M$ and eventually answer the question of when $\widehat{R} = k[[t^{a_1}, t^{a_2}, t^{a_3}]]$ is an almost Gorenstein ring. Throughout, we assume that $T$ is not a Gorenstein ring.

Let $U = k[X, Y, Z]$ be the polynomial ring and regard $U$ as a $\mathbb Z$--graded ring with $U_0=k$, $\operatorname{deg} X=a_1$, $\operatorname{deg}Y =a_2$, and $\operatorname{deg}Z =a_3$. Let $$\varphi:U \to T$$ be the $k$--algebra map defined by $\varphi(X)=t^{a_1}$, $\varphi(Y)=t^{a_2}$, and $\varphi(Z)=t^{a_3}$. Hence $\operatorname{Im} \varphi = T$. Let $$c=\rmc(H) := \min \{n \in \Bbb Z \mid \ell \in H,~\text{if} ~\ell \in \Bbb Z~~\text{and}~~\ell \ge n \}$$
be the conductor of $H$ and $f= \rmf(H):=c-1$ the Frobenius number of $H$. Hence $c \ge 2$, because $T$ is not a Gorenstein ring.  We put $J = \operatorname{\Ker}\varphi$. Then because $T$ is not a Gorenstein ring, thanks to \cite{H}, the ideal $J$ is generated by the maximal minors of the matrix  
$$\begin{pmatrix}   
X^{\alpha} & Y^{\beta} & Z^{\gamma} \\
Y^{\beta'} & Z^{\gamma'} & X^{\alpha'}
\end{pmatrix}$$
where $0< \alpha, \beta, \gamma, \alpha', \beta', \gamma'\in \mathbb Z$. 
Let us  call this matrix the Herzog matrix of $H$.

Let $\Delta_1 = Z^{\gamma+\gamma'}-X^{\alpha'}Y^{\beta}$, $\Delta_2=X^{\alpha+\alpha'}-Y^{\beta'}Z^{\gamma}$, and $\Delta_3=Y^{\beta+\beta'}-X^{\alpha}Z^{\gamma'}$. Then $J = (\Delta_1,\Delta_2, \Delta_3)$ and thanks to the theorem of Hilbert--Burch \cite[Theorem 20.15]{E}, the graded ring $U/J$ possesses a graded minimal free resolution of the  form
$$0\longrightarrow \begin{matrix} U(-\ell)\\ \oplus\\ U(-n)\end{matrix} \overset{\left[ \begin{smallmatrix}
X^{\alpha} & Y^{\beta'}\\
Y^{\beta} & Z^{\gamma'}\\
Z^{\gamma} & X^{\alpha'}
\end{smallmatrix} \right]}{\longrightarrow  } \begin{matrix} U(-d_1)\\ \oplus\\ U(-d_2)\\ \oplus\\ U(-d_3)\end{matrix} \overset{\left[\Delta_1~\Delta_2~\Delta_3\right]}{\longrightarrow } U\overset{\varepsilon }{\longrightarrow } U/J\longrightarrow 0,$$
where $d_1 = \deg\Delta_1 = a_3(\gamma + \gamma')$, $d_2 = \deg\Delta_2 = a_1(\alpha + \alpha')$, $d_3 = \deg\Delta_3 = a_2(\beta + \beta')$, $\ell= a_1\alpha + d_1 = a_2\beta + d_2 = a_3\gamma + d_3$, and $n = a_1\alpha' + d_3 = a_2\beta' + d_1 = a_3\gamma' + d_2$. Therefore $$n - \ell = a_2\beta' -a_1\alpha = a_3\gamma' - a_2\beta = a_1\alpha' -a_3\gamma.$$

Let $\rmK_U = U(-d)$ denote the graded canonical module of $U$ where $d=a_1 + a_2 + a_3$. Then, taking $\rmK_U$--dual of the above resolution, we get the  presentation
\begin{equation*}
\begin{matrix} U(d_1-d)\\ \oplus\\ U(d_2-d)\\ \oplus\\ U(d_3-d)\end{matrix}\overset{\left[ \begin{smallmatrix} X^{\alpha} & Y^{\beta} & Z^{\gamma} \\
Y^{\beta'} & Z^{\gamma'} & X^{\alpha'}
\end{smallmatrix} \right]}{\longrightarrow }\begin{matrix} U(\ell-d)\\ \oplus\\ U(n-d)\end{matrix}\overset{\varepsilon }{\longrightarrow }\rmK_{U/J}\longrightarrow 0 \tag{$\sharp$}
\end{equation*}
of the graded canonical module $\rmK_{U/J} = \Ext_U^2(U/J,\rmK_U)$ of $U/J$.

Let $V=-(\mathbb Z\setminus H) ~(=\{-v \mid v \in \Bbb Z ~\text{but}~v \not\in H\})$. Hence $-f = \operatorname{min} V$. We put $$L = \sum_{v\in V}kt^v$$ in $k[t,t^{-1}]$. Then $L$ is a graded $T$--submodule of $k[t,t^{-1}]$ and $\rmK_T \cong L$ as graded $T$--modules, because  $L$ is the graded $k$--dual of the graded local cohomology module $$\rmH_M^1(T) = k[t,t^{-1}]/T$$ of $T$ with respect to $M$ (\cite[Proposition (2.1.6)]{GW}). We put  $$K := t^f L \cong L(-f).$$ Then $$T\subseteq K\subseteq k[t]=\overline{T} \ \ \text{and} \ \ K/T \cong (L/Tt^{-f})(-f)$$ as graded $T$--modules.

Because $T = U/J$, we are able to combine the data on $K=t^fL$ and $\rmK_{U/J}$. Since $K \cong \rmK_{U/J}(-f)$ and  $$\rmK_{U/J} = U\varepsilon\left(\binom{1}{0}\right) + U\varepsilon\left(\binom{0}{1}\right)$$ with $\deg \left(\varepsilon\left(\binom{1}{0}\right)\right) = d - \ell$ and $\deg\left( \varepsilon\left(\binom{0}{1}\right)\right) = d - n$ (see presentation $(\sharp)$), we get 
$$K = T\zeta + T\eta $$
 for some $\zeta \in K_{d-\ell+f}$ and $\eta \in K_{d - n+f}$, where $K_i$ denotes, for each $i \in \Bbb Z$, the homogeneous component of $K$ with degree $i$.  Hence $n \ne \ell$, because $\mu_T(K)=2$ and $\operatorname{dim}_kK_i \le 1$ for all $i \in \Bbb Z$.

Remember now that $K_0 \ne (0)$ but $K_i = (0)$ for all $i < 0$. Therefore, if $\ell > n$, then $d -n+f>d - \ell +f= 0$. Consequently we get  $f = \ell -d$ and hence $$K = T + Tt^{\ell-n} \ \ \text{and}\ \ T[K] = T[t^{\ell-n}].$$ We similarly have $$K = T + Tt^{n-\ell}\ \ \text{and} \ \ T[K] = T[t^{n-\ell}],$$  if $\ell<n$.

Let $G = U(\ell-d) \oplus U(n-d)$. If $\ell>n$, then we get by presentation $(\sharp)$ above that $K/T \cong \left[G/(U\binom{1}{0} + U\binom{X^{\alpha}}{Y^{\beta'}} + U\binom{Y^{\beta}}{Z^{\gamma'}} + U\binom{Z^{\gamma}}{X^{\alpha'}})\right](-f) \cong \left[U/(X^{\alpha'}, Y^{\beta'}, Z^{\gamma'})\right](n-\ell)$ as graded $T$--modules (remember that $-f=d-\ell$ and that $K/T \cong (L/Tt^{-f})(-f)$ as graded $T$--modules). Hence $$\ell_T(K/T) = \alpha'\beta'\gamma'.$$ Similarly, if $n>\ell$, then $K/T \cong \left[U/(X^{\alpha}, Y^{\beta}, Z^{\gamma})\right](\ell-n)$ as graded $T$--modules and $$\ell_T(K/T) = \alpha\beta\gamma.$$

We are in a position to summarize these arguments.

\begin{thm} \label{algor3} Let $b=|\ell-n|$ be the absolute value of $\ell-n$. We put $I = (t^c, t^{b+c})T$ and $Q = t^cT$, where $c = \rmc (H)$ is the conductor of $H$. Then the following assertions hold true.
\begin{enumerate}[{\rm (1)}]
\item $I$ is a graded canonical ideal of $T$ with  $Q$ a reduction.
\item if $\ell>n$~$(resp. ~\ell<n)$, then $\ell_T(I/Q)=\alpha'\beta'\gamma'$ ~$(resp. ~\ell_T(I/Q)=\alpha\beta\gamma)$. 
\item $\rme_1(IT_M)=\#(H'\setminus H)$ where $H' = \left<a_1, a_2, a_3, b \right>$.
\end{enumerate}
\end{thm}

\begin{proof} Since $I=t^cK\subsetneq T$ and $T \subseteq K \subseteq k[t] = \overline{T}$, we get $I \cong \rmK_T (-(f+c))$ as graded $T$--modules and $Q$ is a reduction of $I$. Hence assertion (1) follows. Assertion (2) is clear, since $I/Q\cong (K/T)(-c)$ as graded $T$--modules. Because $\e_1(IT_M)=\ell_T (T[K]/T)$ by Lemma  \ref{calce1} and $T[K]= T[t^b]$, we get assertion (3). 
\end{proof}

\begin{cor}[cf. {\cite[Corollary 3.3]{NNW2}}]\label{NHW} The ring $k[[t^{a_1}, t^{a_2}, t^{a_3}]]$ is an almost Gorenstein ring if and only if the Herzog matrix is either $\begin{pmatrix}
X & Y & Z \\
Y^{\beta'} & Z^{\gamma'} & X^{\alpha'}
\end{pmatrix}$ or $\begin{pmatrix}
X^{\alpha} & Y^{\beta} & Z^{\gamma} \\
Y & Z & X
\end{pmatrix}$.
\end{cor}

\begin{proof}
Thanks to Proposition \ref{ineq} and Theorem \ref{algor}, the ring $R=T_M$ is an almost Gorenstein ring if and only if $\ell_T(I/Q)=\rmr(T_M)-1=1$. Thus the assertion follows from  Proposition \ref{ff}  and Theorem \ref{algor3} (2), since $k[[t^{a_1}, t^{a_2}, t^{a_3}]] = \widehat{R}$.
\end{proof}

Let us explore two examples.

\begin{ex} \begin{enumerate}[{\rm (1)}]\label{4.3}
\item Let $H=\left<3, 7, 8\right>$ and $T= k[t^3, t^7, t^8]$.  The Herzog matrix of $H$ is 
$\begin{pmatrix}
X^{2} & Y & Z \\
Y & Z & X^{2}
\end{pmatrix},$ so that $k[[t^3, t^7, t^8]]$ is not an almost Gorenstein ring. We have $c=6$, $a_1\alpha=6$, and $a_2\beta'=7$. Hence $b = 1$, so that $I=(t^6, t^7)T$, $Q=t^6T$, and $\ell_T(I/Q)=2$. Since $H'=\mathbb N$, we get  $\e_1(IT_M)=\#H' = \#(\mathbb N\setminus H) = 4$.

\item Let $q>0$ be an integer and $H=\left<4, 4q+3, 4q+5\right>$. The Herzog matrix of $H$ is $\begin{pmatrix}
X^{2q+1} & Y^2 & Z \\
Y & Z & X
\end{pmatrix}$ and  $k[[t^4, t^{4q + 3}, t^{4q + 5}]]$ is an almost Gorenstein ring. We have $c = 4q + 3$, $a_1\alpha=8q+4$, and $a_2\beta'=4q+3$. Hence $b = 4q + 1$, so that  $I = (t^{4q + 3}, t^{8q+4})T$ and  $Q = t^{4q+3}T$, where $T = k[t^4, t^{4q + 3}, t^{4q + 5}]$.
\end{enumerate}
\end{ex}

\section{Gorensteiness in the algebra $\m : \m$}
Let $R$ be a Cohen-Macaulay local ring with maximal ideal $\m$ and $\operatorname{dim} R = 1$. In this section we shall settle in full generality the problem of when the endomorphism algebra $\m : \m$ of $\m$ is a Gorenstein ring. This is the question which V. Barucci and R. Fr{\"o}berg \cite[Proposition 25]{BF} tried to answer in the case where the rings $R$ are analytically unramified and Barucci \cite{B} eventually felt that there was a gap in their proof.

Let $v(R) = \mu_R(\m)$ denote the embedding dimension of $R$ and $e(R) = \rme_0({\m})$ the multiplicity of $R$ with respect to $\m$. We then have the following.

\begin{thm}\label{gorm:m} The following conditions are equivalent.
\begin{enumerate}[{\rm (1)}]
\item $\m:\m$ is a Gorenstein ring.
\item $R$ is an almost Gorenstein ring and $v(R) = e(R)$.
\end{enumerate}
\end{thm}

\begin{proof}
After enlarging the residue class field of $R$, by Proposition \ref{f} we may assume that the field $R/\m$ is  algebraically closed and the ring $\rmQ(\widehat{R})$ is Gorenstein. We may also assume that $R$ is not a discrete valuation ring.  Therefore we have an $R$--submodule $K$ of $\rmQ (R)$ such that $R \subseteq K \subseteq \overline{R}$ and $K \cong \rmK_R$ as $R$--modules (Corollary \ref{existK}). Let us  maintain the same notation as in Setting \ref{setting}. Hence  $S=R[K]$ and $\fkc = R:S$. Let $A=\m:\m$.

(1) $\Rightarrow$ (2)
 Since $R$ is not a discrete valuation ring, $R \subsetneq R:\m = \m:\m=A$ by Lemma \ref{DVR}. Suppose that $R$ is a Gorenstein ring. Then, since $\m = R:A$ and $A$ is a Gorenstein ring, the $A$--module $\m$ is locally free of rank one (\cite[Satz 5.12]{HK}), so that $\m \cong A$ as $A$--modules. Hence $\m = aA$ for some $a\in \m$. Therefore $\m^2 = a\m$, i.e., $v(R) = e(R)$ (see \cite{S1}).

Suppose now that $R$ is not a Gorenstein ring. Since $R:A = \m$, we have  $$A \subseteq K :\m \subseteq S$$ by Corollary \ref{notgor} (1). 

\begin{claim}
Let $X$ be a finitely generated $A$--submodule of $\rmQ (R)$ such that $\rmQ (R) {\cdot}X = \rmQ (R)$. Then $X$ is a reflexive $R$--module, i.e., $X = R:(R:X)$.
\end{claim}

\begin{proof}[Proof of Claim 2]
Notice that $R : (R : A) = R : \m = A$ and $A:(A:X) = X$ for every fractional ideal $X$ of $A$ (\cite[Bemerkung 2.5]{HK}), since $A$ is a Gorenstein ring. We write $A: X = \sum_{i=1}^\ell Ay_i$ where $y_i's$ are units of $\rmQ (R)$. Then $$X = A : (A: X) = A:\sum_{i=1}^\ell Ay_i = \bigcap_{i=1}^\ell A \frac{1}{y_i}.$$ Therefore, because $A\frac{1}{y_i} \cong A$ and $A$ is $R$--reflexive,  we get
\begin{eqnarray*}
X&\subseteq& R : (R: X)\\
&=&R : (R: \bigcap_{i=1}^\ell A\frac{1}{y_i})\\
&\subseteq& R:\sum_{i=1}^\ell (R:A\frac{1}{y_i})\\
&\subseteq&\bigcap_{i=1}^\ell \left[R: (R : A\frac{1}{y_i})\right]\\
&=&\bigcap_{i=1}^\ell A\frac{1}{y_i}\\
&=&X,
\end{eqnarray*}
so that $X$ is a reflexive $R$--module.
\end{proof}

\begin{claim}
$\ell_A(S/A) = \ell_A(\m / \fkc)$.
\end{claim}

\begin{proof}[Proof of Claim 3]
Let $\ell = \ell_A(S/A)$ and take a composition series
$$
A=A_0 \subsetneq A_1 \subsetneq \cdots \subsetneq A_\ell = S
$$
as $A$--modules. Then applying $[R:*]$, by Claim 2 we get
$$
\m = R:A = R:A_0 \supsetneq R:A_1 \supsetneq \cdots \supsetneq R:A_\ell = R:S = \fkc.
$$ Hence $\ell_A(\m / \fkc) \ge \ell$ and we get $\ell_A(S/A) = \ell_A(\m / \fkc)$ by symmetry.
\end{proof}
We now notice that $\ell_A(X) = \ell_R(X)$ for every $A$--module $X$ of finite length, because $A$ is a module-finite extension of $R$ and the field $R/\m$ is algebraically closed. Consequently
$$\ell_R(S/A) = \ell_A(S/A) = \ell_A(\m / \fkc) = \ell_R(\m / \fkc)$$
and therefore by Lemma \ref{DVR} (2) we get
\begin{align*}
\ell_R(S/R) &= \ell_R(S/A) + \ell_R(A/R)\\
&= \ell_R(\m / \fkc) + \ell_R((R:\m) / R)\\
&= (\ell_R(R/\fkc) -1) + \rmr(R)\\
&= \ell_R(R/\fkc) + (\rmr(R) -1),
\end{align*}
so that $\ell_R(I/Q) = \rmr(R) - 1= \mu_R(I/Q)$ by  Lemma \ref{3.4} (3). 
Thus $R$ is an almost Gorenstein ring. Since $\m S \subseteq R$, we have $S \subseteq R : \m = A$. Hence $S = A$ and $\fkc^2 = a\fkc$ for some $a \in \fkc$ by Corollary \ref{notgor} (2). Thus $v(R) = e(R)$, because $\fkc = \m$.

(2) $\Rightarrow$ (1)
Suppose that $R$ is a Gorenstein ring. Then $e(R) \le 2$ and hence every finitely generated $R$--subalgebra of $\overline{R}$ is a Gorenstein ring. In particular,  the ring $A=\m : \m$ is  Gorenstein. Suppose that $R$ is not a Gorenstein ring. Then $S = \m : \m$ by Theorem  \ref{algorcor} and $S$ is a Gorenstein ring by Corollary \ref{notgor} (2), because $\fkc = \m$ and $\m^2 = a\m$ for some $a \in \m$.
\end{proof}

\begin{rem}
In the proof of (1) $\Rightarrow$ (2) of Theorem \ref{gorm:m} the critical part is the fact that $\ell_R(S/A) = \ell_R(\m/\fkc)$, which in our context we safely get by the assumption that $R/\m$ is an algebraically closed filed. Except this part the above proof is essentially the same as was given by Barucci and Fr{\"o}berg \cite{BF}. We nevertheless do not know whether we can still assume that $R/\m$ is an algebraically closed field, even if we restrict the notion of almost Gorenstein ring within the rings which are analytically unramified. See Question \ref{question}.   
\end{rem}

We note an example.

\begin{ex} \label{alex}Let $A$ be a regular local ring with maximal ideal $\n$ and $\operatorname{dim} A = n \ge 3$. Assume that the field $A/\n$ is infinite. Let ${\n} = (X_1, X_2, \ldots, X_n)$ and put $$R = A/\bigcap_{i=1}^{n}(X_1,\ldots,\check{X_i},\ldots, X_n).$$  Let $\m = (x_1, x_2, \ldots,x_n)$ be the maximal ideal of $R$, where $x_i$ denotes the image of $X_i$ in $R$. Then $I = \left(x_i+x_{i+1}\mid 1\leq i\leq n-1\right)$ is a canonical ideal of $R$. The ring $R$ is an almost Gorenstein ring with $\rme_1(I) = \rmr(R) = n-1$. We have $\m : \m = \overline{R}$ and hence $\m : \m$ is a Gorenstein ring.
\end{ex}

\begin{proof} We put $\p_i=(x_1, \ldots, \check{x_i}, \ldots, x_n)$ in $R$. Then $\overline R = \prod_{i=1}^nR/\p _i$ and we have the exact sequence
$$0\to R\overset{\varphi}{\to} \overline R \to C \to 0$$
of $R$--modules, 
where $\varphi(a)=(\overline a, \overline a, \ldots,\overline a)\in \overline R$ for every $a\in R$. For each $1\leq j\leq n$ let  ${\mathbf{e}}_j= (0,\ldots,0,\overset{j}{\check{1}},0,\ldots,0)$ and $\mathbf{e}=\sum_{j=1}^n{\mathbf{e}}_j$. Then $\overline R= R{\mathbf{e}} + \sum_{j=1}^{n-1}R{\mathbf{e}}_j$. Since 
$$
x_i{\mathbf{e}}_j = \left\{
\begin{array}{lc}
0 & (i\not=j),\\
\vspace{1mm}\\
x_i{\mathbf{e}} & (i=j),
\end{array}
\right.
$$
for all $1\leq i, j \leq n$, we get ${\m} \overline R \subseteq \varphi(R)$ and hence $\m = R: \overline R$. Therefore $R$ is an almost Gorenstein ring by Corollary \ref{3.12}. Since $\m : \m = \overline{R}$, $\m : \m$ is a Gorenstein ring. Similarly as in Example \ref{notexist}, $I = \left(x_i+x_{i+1}\mid 1\leq i\leq n-1\right)$ is a canonical ideal of $R$ with $\mu_R(I) = n-1$, so that $R$ is not a Gorenstein ring. Let $Q=(a)$ be a reduction of $I$ and put $K=\frac{I}{a}$ and $S = R[K]$. Then $S = \overline{R}$, since $S = \m : \m$ by Theorem \ref{algorcor}.
\end{proof}

The following  example shows that $R$ is not necessarily an almost Gorenstein ring, even if $S = R[K]$ is a Gorenstein ring.

\begin{ex}  Let $k[[t]]$ be the formal power series ring over a field $k$.
We look at the local ring $R=k[[t^3, t^7, t^8]]$. Then $I = (t^6, t^7)$ is a canonical ideal of $R$ with $(t^6)$ a reduction (see Example \ref{4.3} (1)). 
Choose $K = \frac{I}{t^6} = R + Rt$. Then $S = R[K]= k[[t]]$, so that $S$ is a Gorenstein ring but $R$ is not an almost Gorenstein ring.
\end{ex}

\section{Almost Gorenstein rings obtained by idealization}
In this section we  explore almost Gorenstein rings obtained by idealization. The purpose is to show how our modified notion of almost Gorenstein ring enriches examples and the theory as well.

Let $R$ be a Cohen-Macaulay local ring with maximal ideal $\m$ and $\operatorname{dim} R= 1$. For each $R$--module $M$ we denote by $R \ltimes M$ the idealization of $M$ over $R$. Hence $R\ltimes M = R \oplus M$ as additive groups and the multiplication in $R\ltimes M$ is given by
$$(a,x)(b,y) = (ab, ay +bx).$$
We then have $\fka := (0) \times M$ forms an ideal of $R\ltimes M$ and $\fka^2 = (0)$. Hence, because   $R \cong (R\ltimes M)/\fka$, $R\ltimes M$ is a local ring with maximal ideal $\m \times M$ and $\operatorname{dim} R \ltimes M = 1$. Remember that $M \cong \rmK_R$ as $R$--modules if and only if $R\ltimes M$ is a Gorenstein ring, provided $M$ is a finitely generated $R$--module and $M \ne (0)$ (\cite{R}).

Let us begin with the following.

\begin{prop} \label{alid1} Let $I$ be an arbitrary $\m$--primary ideal of $R$ and suppose that there exists an element $a \in I$ such that $Q = (a)$ is a reduction of $I$. Assume that $R$ possesses the canonical module $\rmK_R$ and put $I^{\vee} = \Hom_R(I,\rmK_R)$.
Then the following conditions are equivalent.
\begin{enumerate}[{\rm (1)}]
\item $R \ltimes I^{\vee}$ is an almost Gorenstein ring.
\item ${\m}I={\m}Q$ and $I^2=QI$.
\item $\mu_R(I/Q) = \ell_R(I/Q) = \rme_1(I)$.
\end{enumerate} 
\end{prop}

\begin{proof} 
By Proposition \ref{ineq} it is enough to show (1)$\Leftrightarrow $(2). Let  $M=I^{\vee}$. We look at the exact sequence
$$0\to R^{\vee} \overset{\iota^{\vee}} \to I^{\vee} \to \Ext_R^1(R/I, \rmK_R) \to 0$$
derived from the exact sequence
$$0\to I \overset{\iota}{\to}  R \to R/I \to 0$$
of $R$--modules, where $\iota : I \to R$ denotes the embedding. We put $L = \operatorname{Im} \iota^{\vee}$. Then $\Ext_R^1(R/I, \rmK_R) \cong M/L$. Hence $I = \Ann_R M/L$, because $$R/I \cong \Ext_R^1(\Ext_R^1(R/I, \rmK_R), \rmK_R)$$ by \cite[Satz 6.1]{HK}.
Let $A = R \ltimes M$ and $T=R \ltimes L$. Then $A$ is a module--finite extension of $T$ with $\rmQ(A) = \rmQ(T)$. We have $\rmK_A \cong T:A$, because $T$ is a Gorenstein ring. A direct computation shows $$T:A = (\operatorname{Ann}_RM/L) \times L = I \times L.$$ Hence $I \times L$ is a canonical ideal of $A$. Let $f = (a,0) \in I \times L$. Then $fA$ is a reduction of $I \times L$, since $Q = aR$ is a reduction of $I$ and $[(0) \times L]^2=(0)$ in $A$. Therefore by Theorem \ref{algor} the ring $A$ is almost Gorenstein if and only if $(\m \times M) {\cdot}(I \times L) \subseteq fA$. The latter condition is equivalent to saying that $$\m I \subseteq Q ~~\text{and}~~IM = QM \supseteq  \m L.$$

Look at the exact sequence
$$
0 \to M \overset{a}{\to} M \to \Ext_R^1(I/QI, \rmK_R) \to 0
$$
obtained by the exact sequence $$
0 \to I \overset{a}{\to} I \to I/QI \to 0.
$$
We then have $\Ann_R M/QM = \Ann_R I/QI$, because $\Ext_R^1(I/QI, \rmK_R) \cong M/QM$. Hence $IM = QM$ if and only if $I^2 = QI$.

Let $\varphi: R \to I$ be the $R$--linear map defined by $\varphi(1) = a$ and look at the exact sequence 
$$0 \to M \overset{\varphi^\vee}{\to} R^\vee \to \Ext_R(I/Q, \rmK_R) \to 0$$ obtained by the exact sequence
$$
0\to R \overset{\varphi} \to I \to I/Q \to 0.
$$
We then have a commutative diagram 
$$
\xymatrix{ 0 \ar[r]& QM \ar[r]  \ar @{} [dr] |{\circlearrowleft}& L \ar[r] & L/QM \ar[r]& 0\\
0 \ar[r]& M \ar[r]^{\varphi^{\vee}} \ar[u]^{\wr} & R^{\vee} \ar[u]_{\wr}^{\iota^{\vee}} \ar[r]& \Ext_R^1(I/Q, \rmK_R) \ar[r]& 0
}
$$
with exact rows, so that  
$$\Ext_R^1(I/Q, \rmK_R)\cong L/QM.$$
Therefore $\m L \subseteq QM$, once $\m I \subseteq Q$. Hence $A = R \ltimes I^\vee$ is an almost Gorenstein ring if and only if $\m I = \m Q$ and $I^2 = QI$, which completes the proof of Proposition \ref{alid1}.
\end{proof}

As an immediate consequence of Proposition \ref{alid1} we get the following.

\begin{cor}\label{5.2} 
Suppose that $R$ possesses the canonical module $\rmK_R$. If $R$ is not a discrete valuation ring, then $R \ltimes (Q:_R\m)^\vee$ is an almost Gorenstein ring for every parameter ideal $Q$ in $R$, where $(Q:_R\m)^\vee = \Hom_R(Q:_R\m, \rmK_R)$.
\end{cor}

\begin{proof}
Let  $I = Q :_R\m $. Then by \cite{CP} we have $I^2=QI$, so that the ideal $I$ satisfies condition (2) in Proposition \ref{alid1}. 
\end{proof}

\begin{thm} \label{alid2} Suppose  that $R$ possesses the canonical module $\rmK_R$ and the residue class field $R/\m$ of $R$ is infinite. Let $M$ be a Cohen-Macaulay $R$--module  with $\operatorname{Ann}_R M = (0)$. Then the following conditions are equivalent.
\begin{enumerate} [{\rm(1)}]
\item $R\ltimes M$ is an almost Gorenstein ring.
\item There exists an ${\m}$--primary ideal $I$ of $R$  and $a \in I$ such that $M \cong I^\vee$ as $R$--modules, $I^2=QI$, and ${\m}I=\m Q$, where $Q = (a)$ and $I^\vee = \Hom_R(I, \rmK_R)$.
\end{enumerate}
\end{thm}

\begin{proof} Thanks to Proposition \ref{alid1}, we have only to prove (1) $\Rightarrow$ (2). It suffices to show $M \cong I^{\vee}$ for some ideal $I$ of $R$. Notice that $\rmQ (\widehat{A})$ is a Gorenstein ring, since the ring $\widehat{A} = \widehat{R} \ltimes \widehat{M}$ is almost Gorenstein (Proposition \ref{ff}), where $\widehat{*}$ denotes the $\m$--adic completion. Let $\fkp \in \Ass \widehat{R}$. Then, since $$(\widehat{R} \ltimes \widehat{M})_{{\p}\times \widehat{M}} = \widehat{R}_{\p} \ltimes \widehat M_{\p}$$ is a Gorenstein ring and $\widehat M_{\p} \ne (0)$, by \cite{R} we get $(\rmK_{\widehat{R}})_\fkp \cong \rmK_{{\widehat{R}}_\fkp} \cong \widehat{M}_\fkp.$ Consequently $$\rmQ(\widehat R)\otimes _{\widehat R} \rmK_{\widehat R} \cong \rmQ(\widehat R)\otimes _{\widehat R} \widehat M$$ as $\rmQ(\widehat{R})$--modules, which yields an exact sequence 
$$0\to \rmK_{\widehat R} \overset{\iota}{\to} \widehat M \to L \to 0$$
of $\widehat R$--modules with $\ell_{\widehat R}(L)<\infty$. Via the homomorphism $\iota$, let us regard $\rmK_{\widehat{R}}$ as an $\widehat{R}$--submodules of $\widehat{M}$. Then, since $\widehat{\m}^{\ell}\widehat M \subseteq \rmK_{\widehat{R}}$ for some  $\ell \gg 0$, we get $\rmK_{\widehat{R}} = \widehat{R}K$ in $\widehat{M}$ where $K = \rmK_{\widehat{R}} \cap M$ (here $M$ is considered to be an $R$--submodule of $\widehat{M}$). Therefore $K = \rmK_R$ and $\ell_R(M/K) < \infty$, because $\widehat{K} = \widehat{R}K = \rmK_{\widehat{R}}$. Hence, taking the $\rmK_R$--dual of the exact sequence $$0 \to \rmK_R \to M \to M/\rmK_R \to 0,$$ we get an embedding $$0\to M^{\vee} \to (\rmK_R)^{\vee}.$$
Thus $M^{\vee}$ is isomorphic to an ideal $I$ of $R$, because $R = \rmK_R^\vee$ (\cite[Satz 6.1]{HK}). This completes the proof of Theorem \ref{alid2}.
\end{proof}

When $R$ is a Gorenstein ring, we can simplify Theorem \ref{alid2} as follows.

\begin{cor}\label{alidg} Assume that $R$ is a Gorenstein ring and let $M$ be a Cohen-Macaulay $R$--module with $\operatorname{Ann}_R M = (0)$. Then the following conditions are equivalent.
\begin{enumerate}[{\rm (1)}]
\item $R\ltimes M$ is an almost Gorenstein ring.
\item Either $M \cong R$ or $M \cong {\m}$.
\end{enumerate}
\end{cor}

\begin{proof} Thanks to Proposition \ref{f} (3) and Proposition \ref{ff}, enlarging the residue class field of $R$, we may assume that the field $R/\m$ is infinite. Let $*^\vee = \Hom_R(*,R)$.

(2) $\Rightarrow$ (1) If $M \cong R$, then $R\ltimes M$ is a Gorenstein ring. Suppose $M \cong \m$. To see that $R\ltimes M$ is an almost Gorenstein ring, we may assume $R$ is not a discrete valuation ring. Since $\m \cong \Hom_R(\Hom_R(\m, R),R) \cong (Q :_R \m)^\vee$ for every parameter ideal $Q$ in $R$ (\cite[Korollar  6.8]{HK}), the ring $R \ltimes M ~(\cong R \ltimes \m)$ is an almost Gorenstein ring by Corollary  \ref{5.2}.

(1) $\Rightarrow$ (2) By Theorem \ref{alid2} $M\cong I^{\vee}$ for some $\m$--primary ideal $I$ of $R$ such that ${\m}I={\m}Q$ and $I^2 = QI$, where $Q=(a)$ is a reduction of $I$. Since $Q \subseteq I \subseteq Q:_R{\m}$ and $\ell_R((Q:_R{\m})/Q) = 1$, we have either $I=Q$ or $I=Q:_R{\m}$. If $I=Q$, then $M\cong R$. If $I=Q:_R{\m}$, then $M \cong (Q :_R \m)^\vee \cong \Hom_R(\Hom_R(\m, R),R) \cong \m$. Hence the result.
\end{proof}

The following result shows the property of being an almost Gorenstein ring is preserved via idealization of the maximal ideal,  and vice versa.

\begin{thm} \label{algorm}The following conditions are equivalent.
\begin{enumerate}[{\rm (1)}]
\item $R\ltimes {\m}$ is an almost Gorenstein ring.
\item $R$ is an almost Gorenstein ring.
\end{enumerate}
When this is the case, $v(R\ltimes {\m}) = 2v(R)$.
\end{thm}

\begin{proof} We may assume that the residue class field $R/\m$ of $R$ is infinite and the ring $\rmQ (\widehat{R})$ is Gorenstein. Hence  there exists an $R$--submodule $K$ of $\rmQ (R)$ such that $R\subseteq K\subseteq \overline R$ and $K \cong \rmK_R$ as $R$--modules. We maintain the same notation as in Setting \ref{setting}. Hence $S = R[K]$. We choose a non-zerodivisor $a\in {\m}$ so that $aK\subsetneq R$. Let $I=aK$,  $Q=(a)$, and $J=I:_R {\m}$. Therefore $I$ is a canonical ideal of $R$.

(1) $\Rightarrow$ (2)
We may assume that $R$ is not a Gorenstein ring. By Theorem \ref{alid2} we may choose an $\m$--primary ideal $\fka$ of $R$ and $b \in \fka$  so that $\fka^2 = b \fka$, $\m \fka = \m b$, and $\m \cong \fka^\vee$, where $*^\vee =\Hom_R(*,\rmK_R)$. Since $\ell_R((I : \m)/I)=1$, we get $J = I : \m  = a(K : \m)$. On the other hand, since $K : \m \subseteq S$ by Corollary \ref{notgor} (1), we get 
$$Q = (a) \subseteq I = aK \subseteq J = a(K : \m) \subseteq aS \subseteq a\overline{R}.$$  Hence $Q$ is also a reduction of $J$. Now notice that $\m \cong J^\vee$, since $J =I : \m \cong \m^\vee$. Then, because $R\ltimes J^\vee$ is an almost Gorenstein ring,  we get $\m J \subseteq Q$ by Proposition \ref{alid1}, so that $\m I \subseteq \m J \subseteq Q$. Hence $R$ is an almost Gorenstein ring.

(2) $\Rightarrow$ (1)
By Corollary \ref{alidg} we may assume that $R$ is not a Gorenstein ring. Choose a regular element $b \in \m$ so that $b S \subsetneq R$ and put $\fka = bS$. Then $b \in \fka$, $\fka^2 = b \fka$, and $\m \fka \subseteq (b)$, since $R$ is an almost Gorenstein ring. Now notice that $S = K :\m \cong \m^\vee$ (Theorem \ref{algorcor}) and we have $\m \cong S^\vee \cong \fka^\vee$. Hence $R \ltimes \m$ is an almost Gorenstein ring by Proposition \ref{alid1}.

Since the maximal ideal of $R \ltimes \m$ is $\m \times \m$, we get 
\begin{eqnarray*}
v(R\ltimes \m)&=&\ell_{R\ltimes \m}(({\m}\times{\m})/({\m}\times{\m})^2)\\
&=& \ell_R(({\m}\oplus {\m})/({\m}^2\oplus {\m}^2))\\
&=& 2\ell_R({\m}/{\m}^2)\\
&=& 2v(R),
\end{eqnarray*}
which proves the last equality.
\end{proof}

We need the following.

\begin{lem}\label{alidm} The following conditions are equivalent.
\begin{enumerate}[{\rm (1)}]
\item $R \ltimes \m$ is a Gorenstein ring.
\item $R$ is a discrete valuation ring.
\end{enumerate}
\end{lem}

\begin{proof}
If $R\ltimes {\m}$ is a Gorenstein ring, then $\m \cong \rmK_R$ as $R$--modules and hence the $R$--module $\m$ has finite injective dimension, which yields that $R$ is a discrete valuation ring (\cite[p.947, Corollary 3]{Bu}). If $R$ is a discrete valuation ring, then $\m \cong R$, so that $R\ltimes \m$ is  a Gorenstein ring.
\end{proof}

Let us note examples of almost Gorenstein rings obtained by idealization.

\begin{ex} Let $n\ge 0$ be an integer. We  put 
$$
 R_n = \left\{
 \begin{array}{ll}
 R & (n=0),\\
 R\ltimes \m & (n=1),\\
 (R_{n-1})_1 & (n>1).
 \end{array}
 \right.
$$ Then the following assertions hold true.
\begin{enumerate}[{\rm (1)}]
\item If $R$ is a Gorenstein ring, then $R_n$ is an almost Gorenstein ring for all $n\ge 0$.
\item $R_n$ is not a discrete valuation ring for every $n\ge 1$. Therefore by Lemma \ref{alidm} $R_{n+1}$ is not a Gorenstein ring for all $n\ge 1$.
\end{enumerate}
\end{ex}

\begin{ex}
The almost Gorenstein ring $R_3$ in Example \ref{ex} is isomorphic to $$k[[X,Y]]/(Y^2) \ltimes (X,Y)/(Y^2),$$ where $k[[X,Y]]$ denotes the formal power series ring over the field $k$.
\end{ex}

We close this paper with  the following.

\begin{cor} \label{Bgor} Suppose that $R$ is an almost Gorenstein ring but not a Gorenstein ring and the residue class field $R/\m$ of $R$ is infinite. Choose an $R$--submodule $K$ of $\rmQ (R)$ so that  $R\subseteq K\subseteq \overline R$ and $K \cong \rmK_R$ as $R$--modules. Let  $S=R[K]$. Then the following conditions are equivalent.
\begin{enumerate}[{\rm (1)}]
\item $S$ is a Gorenstein ring.
\item $R\ltimes S$ is an almost Gorenstein ring.
\item $v(R) = e(R)$.
\item $v(R\ltimes S) = e(R\ltimes S)$.
\end{enumerate}
When this is the case, $S = \m : \m$, $R\ltimes S$ is not a Gorenstein ring, and   $v(R\ltimes S ) = 2v(R)$.
\end{cor} 

\begin{proof} We get by Theorem \ref{algorcor} that $S = K:{\m} \cong {\m}^{\vee} := \Hom_R(\m,\rmK_R)$, while $S = \m : \m$. Therefore  Theorem \ref{gorm:m} (resp. Proposition \ref{alid1}) shows (1) $\Leftrightarrow$ (3) (resp. (2) $\Leftrightarrow$ (3)). Let $a \in \m$ such that $Q=(a)$ is a reduction of $\m$ and put $A = R \ltimes S$. Then $f = (a,0) \in A$ generates a reduction of the maximal ideal $\n := \m \times S$ of $A$, so that $v(A) = e(A)$ if and only if $\n^2 = f \n$, that is  $\m^2 = Q\m$ and $\m S = QS$. Since $S \cong \m^\vee$,  by Proof of Proposition \ref{alid1} we have $\m S = QS$, once $\m^2 = Q\m$.  Hence  we get  the equivalence (3) $\Leftrightarrow$ (4).

Let us check the last assertion. We have $S = \m :\m$, as we have mentioned above. Suppose that  $R \ltimes S$ is a Gorenstein ring. Then $S \cong \rmK_R$ (\cite{R}). Hence $\m \cong S^\vee \cong \rmK_R^\vee \cong R$ (\cite[Satz 6.1]{HK}). This is impossible, because  $R$ is not a Gorenstein ring.
To see $v(A) = 2v(R)$, notice that $\n^2= \m^2 \times \m S$. We then have 
\begin{eqnarray*}
v(A) &=& \ell_A(\n/\n^2)\\
&=& \ell_R(\n/\n^2)\\
&=& \ell_R((\m \oplus S)/(\m^2 \oplus \m S))\\
&=& v(R) + \mu_R(S),
\end{eqnarray*}
while because $S \cong \m^\vee$ and $\m^2 = Q\m$, by \cite[Bemerkung 1.21, Satz 6.10]{HK} we get 
$$\mu_R(S) = \rmr_R(\m) = v(R),$$
where $\rmr_R(\m) = \ell_R(\Ext_R^1(R/\m,\m))$ denotes the Cohen-Macaulay type of the $R$--module $\m$. Hence  $v(A) = 2v(R)$ as required.
\end{proof}

\begin{ex} Suppose that $R$ possesses the canonical module $\rmK_R$. For each $n\ge 0$ let  
$$
 R_n = \left\{
 \begin{array}{ll}
 R & (n=0),\\
 R\ltimes \m^\vee& (n=1),\\
 (R_{n-1})_1 & (n>1),
 \end{array}
 \right.
 $$
where $\m^\vee = \Hom_R(\m,\rmK_R)$. We assume that $R$ is an almost Gorenstein ring but not a Gorenstein ring and that the residue class field $R/\m$ of $R$ is infinite. Then the following assertions hold true.
\begin{enumerate}[{\rm (1)}] 
\item $R_n$ is an almost Gorenstein ring with $v(R_n)=2^nv(R)$ for all  $n\ge 0$.
\item $R_n$ is not a Gorenstein ring for any $n \ge 0$.
\end{enumerate}
\end{ex}

\begin{proof}
Choose an $R$--submodule $K$ of $\rmQ (R)$ so that  $R\subseteq K\subseteq \overline R$ and $K \cong \rmK_R$ as $R$--modules and put  $S=R[K]$. 
Then, since $\m^\vee \cong S$ by Theorem \ref{algorcor},  the assertions immediately follow from Corollary \ref{Bgor} by induction on $n$.  
\end{proof}






\end{document}